\documentclass[11pt]{article}

    \usepackage{url}
    \usepackage{verbatim}
    \usepackage[titletoc]{appendix}
    \usepackage{graphicx}
    \usepackage{hyperref}
    \usepackage{bm}
	\usepackage{amsthm}
	\usepackage{amsmath}
	\usepackage{amsfonts}
	\usepackage{graphicx}
    \usepackage{nicefrac}
    \usepackage{longtable}
    \usepackage{color}
    \usepackage{graphicx, amssymb,graphics}
    \usepackage{algorithmicx,algorithm}
    \usepackage{paralist}

\makeatletter 
\let\myorg@bibitem\bibitem
\def\bibitem#1#2\par{%
	\@ifundefined{bibitem@#1}{%
		\myorg@bibitem{#1}#2\par
	}{%
		\begingroup
		\color{\csname bibitem@#1\endcsname}%
		\myorg@bibitem{#1}#2\par
		\endgroup
	}%
}

\newtheorem{rem}{Remark}[section]
\newtheorem{thm}{Theorem}[section]

\newcommand{\hf}{\nicefrac{1}{2}}

\newcommand{\reff}[1]{{\rm (\ref{#1})}}
	\newcommand\be {\begin{equation}}
	\newcommand\ee {\end{equation}}

\newcommand{\R}{\mathbb{R}}            
\newcommand{\bD}{\bm{D}}
\newcommand{\bE}{\bm{E}}
\newcommand{\ve}{\varepsilon}          

\numberwithin{equation}{section}

\graphicspath{{./Figs/}}

\title{Structure-preserving numerical method for \\ Maxwell-Amp\`{e}re Nernst-Planck model}

\date{\today}

\begin{document}
	
	\author{
Zhonghua Qiao\thanks{Department of Applied Mathematics and Research Institute for Smart Energy, The Hong Kong Polytechnic University, Hung Hom, Hong Kong. (zhonghua.qiao@polyu.edu.hk)},~
\and
Zhenli Xu\thanks{School of Mathematical Sciences, MOE-LSC, CMA-Shanghai and Shanghai Center for Applied Mathematics, Shanghai Jiao Tong University, Shanghai 200240, China. (xuzl@sjtu.edu.cn)},~
\and
Qian Yin\thanks{Corresponding author. School of Mathematical Sciences, MOE-LSC and CMA-Shanghai, Shanghai Jiao Tong University, Shanghai 200240, China. (sjtu\_yinq@sjtu.edu.cn)},~
\and
Shenggao Zhou\thanks{School of Mathematical Sciences, MOE-LSC, CMA-Shanghai and Shanghai Center for Applied Mathematics, Shanghai Jiao Tong University, Shanghai 200240, China. (sgzhou@sjtu.edu.cn)}
}

 \maketitle

\begin{abstract}	
Charge dynamics play essential role in many practical applications such as semiconductors, electrochemical devices and transmembrane ion channels.  A Maxwell-Amp\`{e}re Nernst-Planck (MANP) model that describes charge dynamics via concentrations and the electric displacement is able to take effects beyond mean-field approximations into account.  To obtain physically faithful numerical solutions, we develop a structure-preserving numerical method  for the MANP model whose solution has several physical properties of importance. By the Slotboom transform with entropic-mean approximations, a positivity preserving scheme with Scharfetter-Gummel fluxes is derived for the generalized Nernst-Planck equations. To deal with the curl-free constraint, the dielectric displacement from the Maxwell-Amp\`{e}re equation is further updated with a local relaxation algorithm of linear computational complexity.  We prove that the proposed numerical method unconditionally preserves the mass conservation and the solution positivity at the discrete level, and satisfies the discrete energy dissipation law with a time-step restriction. Numerical experiments verify that our numerical method has expected accuracy and structure-preserving properties. Applications to ion transport with large convection, arising from boundary-layer electric field and Born solvation interactions, further demonstrate that the MANP formulation with the proposed numerical scheme has attractive performance and can effectively describe charge dynamics with large convection of high numerical cell P\'{e}clet numbers.

\bigskip

\noindent
{\bf Key words and phrases}: Poisson-Nernst-Planck equations; 
Maxwell--Amp\`{e}re Nernst--Planck equations;  Positivity preserving; Energy dissipation; Local curl-free algorithm;  Convection dominated problem

\noindent
\end{abstract}

\section{Introduction}

Charge transport is ubiquitous in various applications, such as semiconductors~\cite{2012semiconductor}, microfluidics~\cite{Schoch:RMP:08}, and biological ion channels~\cite{E:CP:98}. It is often described by the classical Poisson-Nernst-Planck (PNP) equations, which consist of Nernst-Planck (NP) equations for the diffusion and convection of the ionic concentrations. The convection is driven by the gradient of electric potential, which in turn is governed by the Poisson's equation. The PNP equations can be derived as the gradient flow of an electrostatic free energy that is based on mean-field approximations of ionic interactions.  However, the mean-field approximations break down in cases where the ionic correlations, steric effects, and dielectric inhomogeneity become nonnegligible.  To address these issues, {various modified PNP models with mathematical analysis have been proposed to achieve decent accuracy in the description of complicated charge transport~\cite{BAO:PRL:1997,HorngLinLiuBob_JPCB12, KBA:PRE:2007, SWZ_CMS18,  LJX:SIAP:2018,LZ:BJ:2011,  MXZ:SIAP:2021, citetwo}}.

Recently, a Maxwell-Amp\`{e}re Nernst-Planck (MANP) framework has been proposed to describe the charge dynamics using ionic concentrations and the electric displacement, rather than the electric potential, as the unknown quantities~\cite{Qiao2022ANP}. The MANP model is inspired by an observation that the motion of charged particles is driven by the electric field or electric displacement, not the electric potential itself. Meanwhile, the evolution of the electric displacement can be described by the Maxwell-Amp\`{e}re equation. In comparison with the PNP formulation, one of the advantages of the MANP model is that numerically solving a Poisson's equation with a variable coefficient can be replaced by performing at each time step a local curl-free relaxation algorithm of linear complexity. The formulation based on concentrations and the electric displacement also facilitates the design of robust methods for solving steady states of charged systems~\cite{BSD:PRE:2009, RM:PRL:2004, ZWL:PRE:2011}. In addition, the numerical solution to the MANP model has been shown to have several important physical properties, such as mass conservation, solution positivity, and free-energy dissipation~\cite{Qiao2022ANP}. Since the physical properties play a crucial role in both theoretical and numerical analysis of the model, it is highly desirable to design numerical methods that are able to guarantee these properties at the discrete level.

{Many efforts have been devoted to the development of structure-preserving numerical methods for the PNP-type equations, 
see e.g., ~\cite{LW:JCP:14, Gao2017JSC, Ding19JCP,  LM2020, Ding2020JCP,  hu2020NM, Metti:JCP:2016, liu2017DG, Qian2021positive, LiuChun2020positivity,LJX:SIAP:2018,SJXJ2021}.} In order to remain the solution positivity, the NP equations are often reformulated in the Slotboom variables. Liu and Wang \cite{LW:JCP:14} have proposed a free-energy satisfying finite difference scheme with conditional positivity preservation, based on the Slotboom transform that is also known as the non-logarithmic Landau formulation for the Fokker-Planck equation.  
Using harmonic-mean approximations of the Slotboom variables on half-grid points, structure-preserving finite difference schemes have been proposed to solve the PNP-type equations~\cite{Ding19JCP, Ding2020JCP}. Using harmonic-mean approximations can effectively overcome the problem of large condition number of the coefficient matrix, which is due to large variations in the exponentially transformed diffusion coefficients in the Slotboom transform. The Slotboom variables have also been used to derive a fully implicit scheme with unconditional structure-preserving properties in \cite{hu2020NM}. The convergence analysis for structure-preserving finite difference methods based on the Slotboom transform has been developed as well~\cite{Ding2022con}. 
Recently, nonlinear implicit schemes based on the energy gradient flow formulation have also been proposed in the works~\cite{LiuChun2020positivity, SJXJ2021}. The existence of a unique positive solution for the nonlinear schemes is proved by using the logarithmic singularity at zero. Rigorous analysis on optimal convergence order has been conducted as well in the work~\cite{LiuChun2020positivity}. 

We propose a structure-preserving scheme for the MANP equations in this work in order to obtain physically faithful numerical solutions.
Entropic-mean approximation of mobility functions on half-grid points is employed to derive the Scharfetter-Gummel fluxes for the NP equations, which are reformulated in the Slotboom variables. The electric displacement is first updated by numerically solving the Maxwell-Amp\`{e}re equation, and then followed by performing a local relaxation algorithm to reach the curl-free property of the electric field. Such a curl-free local algorithm was first developed to accomplish efficient long-range Coulombic interactions in molecular simulations~\cite{M:JCP:2002, MR:PRL:2002, PD:JPCM:2004,RM:JCP:2004MC, RM:PRL:2004, FH:PRE:2014,FXH:JCP:14}, and then extended to numerically solving the Poisson-Boltzmann type of equations~\cite{BSD:PRE:2009, ZWL:PRE:2011}. The curl-free property is achieved at the discrete level once the relaxation converges. Meanwhile, the Gauss's law is maintained during the local relaxation. This idea is based on the constraint minimization of a potential energy, whose minimizer has explicit expression in every iteration, and therefore it is of linear computational complexity. It is important to note that the relaxation approach can treat dielectric inhomogeneity naturally.  Therefore, the MANP formulation, coupled with the local relaxation algorithm, can efficiently avoid solving a Poisson's equation with variable coefficient, which is often the time-consuming step in semi-implicit time-stepping for the PNP formulations. Detailed numerical analysis is performed to show that the whole numerical algorithm can unconditionally preserve mass conservation and solution positivity at the discrete level, and keeps the discrete energy dissipation with a time-step restriction.  Extensive numerical tests are conducted to demonstrate that the numerical method has expected numerical accuracy, efficiency, and structure-preserving properties. Applications to ion transport with large convection, arising from boundary-layer electric field and Born solvation interactions, illustrate that the MANP formulation with the proposed numerical scheme can effectively handle large convection with high numerical cell P\'{e}clet numbers.

The rest of this paper is organized as follows. In Section~\ref{s:MANP}, we present the MANP model and show some intrinsic properties of the MANP model. We develop the numerical method for solving the MANP model and give the corresponding numerical analysis on structure-preserving properties of the proposed numerical scheme in Section~\ref{s:NumMethod}. Numerical tests are carried out to demonstrate the effectiveness of the algorithm in Section~\ref{s:NumRes}. The paper ends with some conclusions in Section~\ref{s:Con}.

\section{Maxwell-Amp\`{e}re Nernst-Planck model}\label{s:MANP}
Consider a charged system occupying a bounded domain $\Omega$ that includes an electrolyte solution of $M$ species of mobile ions and fixed charges.  The electrostatic free energy of such a system in a dimensionless form is given by
\begin{equation}\label{Energy}
F[c^1, \dots, c^M]=\int_\Omega \left[\kappa^2\ve |\nabla\phi|^2 + \sum_{\ell=1}^{M} c^\ell\left(\log c^\ell +\mu^{\ell,\text{cr}} \right)\right] d\bm{r},
\end{equation}
where $c^\ell=c^\ell(\bm{r},t)$ is the ionic concentration for the $\ell$-th species ($\ell=1,\dots, M$) at time $t$, $\phi=\phi(\bm{r},t)$ is the electric potential, $\ve(\bm{r})$ is the space-dependent relative dielectric coefficient, and $\kappa$ is a dimensionless parameter defined as the ratio of the Debye length to characteristic length. The electric potential is governed by the Poisson's equation
\[
-\nabla \cdot 2\kappa^2\varepsilon \nabla \phi =\rho,
\]
where the total charge density $\rho=\sum_{\ell=1}^M q^{\ell} c^{\ell}+\rho^f$, $q^{\ell}$ is the ionic valence, and $\rho^f=\rho^f(\bm{r})$ is the distribution of fixed charges. Here, $\mu^{\ell,\text{cr}}$ is the excess chemical potential accounting for ionic correlation~\cite{BAO:PRL:1997,LJX:SIAP:2018,MXZ:SIAP:2021,Wu:JCP:2002,Wang:PRE:2010,Qing:Langmuir2020} beyond the mean-field approximation. To include steric effects and Born solvation interactions, for instance, we consider
\begin{equation}
\mu^{\ell,\text{cr}}= -\frac{v^{\ell}}{v^0} \log(v^0c^0)+  \frac{\chi  (q^\ell)^2}{a^\ell}    \left(\frac{1}{\varepsilon} -1\right) ,
\end{equation}
where $c^0=\left(1-\sum_{\ell=1}^M v^{\ell}c^{\ell}\right)/v^0$ is the concentration of solvent molecules, $v^{\ell}$ and $v^0$ represent the volumes of ions of the $\ell$-th species and solvent molecules, $\chi$ is a dimensionless parameter, and $a^{\ell}$ is the Born radius for ions of the $\ell$-th species~\cite{MXZ:SIAP:2021, Qiao2022ANP}.

The dynamic of ions is governed by the conservation law
\[
\frac{\partial c^{\ell}}{\partial t}=-\nabla \cdot \bm{J}^{\ell},
\]
with the ionic flux given by
$$\bm{J}^{\ell}=- \kappa c^{\ell} \nabla \mu^{\ell}.$$
Here the total chemical potential is given by the first variation of the free energy, i.e.,
\[
	\mu^{\ell}= \log c^{\ell}+q^{\ell}\phi +\mu^{\ell,\text{cr}}.
\]
Therefore, the ion dynamic is described by modified Nernst-Planck {equations}
\begin{equation}\label{NP}
\frac{\partial c^{\ell}}{\partial t}=\nabla \cdot \kappa \left( \nabla c^\ell+q^\ell c^\ell \nabla \phi+c^\ell \nabla \mu^{\ell,\text{cr}}\right), ~\ell =1, \cdots, M.
\end{equation}

It is seen from \reff{NP} that the convection of ions is driven by the gradient of electric potential, i.e., the electric field. An alternative view to describe charge dynamics is based on the description via the electric displacement $\bD=\ve \bE$, where $\bE=-\nabla \phi$ is the electric field. The charge conservation reads
\begin{equation}
	\frac{\partial \rho}{\partial t}+\nabla\cdot \bm{j} =0,
\end{equation}
with the charge density current given by
\begin{equation}
\bm{j}=-\sum_{\ell=1}^{M} \kappa q^\ell c^\ell \nabla \mu^\ell.
\end{equation}
Application of the Gauss's law $\nabla \cdot 2\kappa^2\bD = \rho$ leads to
\begin{equation}
	\nabla\cdot \left(\frac{\partial \bm{D}}{\partial t} +\frac{\bm{j}}{2\kappa^2} \right)=0,\label{Gs}
\end{equation}
which gives the Maxwell--Amp\`{e}re law
\begin{equation}
	\frac{\partial \bm{D}}{\partial t} +\frac{\bm{j}}{2\kappa^2}= \bm{\Theta}.
\end{equation} 
Here $\bm{\Theta}$, as an additional degree of freedom to enforce the curl-free condition $\nabla \times (\bm{D}/\varepsilon)=\bm{0}$,  satisfies the Coulomb gauge condition $\nabla \cdot \bm{\Theta}= 0$. It has been shown that the curl free condition guarantees the existence of the electric potential satisfying the Poisson's equation~\cite{Qiao2022ANP}.  In summary, the charge dynamics can be described by the following Maxwell--Amp\`{e}re Nernst--Planck equations
\begin{equation}\label{MANP}
\left\{
\begin{aligned}
& \frac{\partial c^{\ell}}{\partial t}=\nabla \cdot \kappa \left(\nabla c^\ell-\frac{q^\ell c^\ell \bm{D}}{\varepsilon}+c^\ell \nabla \mu^{\ell,\text{cr}}\right),~  \ell =1, \cdots, M, \\
&\frac{\partial \bm{D}}{\partial t} =\sum_{\ell=1}^M \frac{q^\ell}{2 \kappa} \left(\nabla c^\ell-\frac{q^\ell c^\ell \bm{D}}{\varepsilon}+c^\ell \nabla \mu^{\ell,\text{cr}}\right) +\bm{\Theta}, \\
&\nabla \cdot \bm{\Theta}=0, \\
& \nabla \times\frac{\bm{D}}{\varepsilon}=\bm{0}.
\end{aligned}
\right.
\end{equation}

{The properties of the MANP model and its equivalence to the PNP model have been thoroughly studied in~\cite{Qiao2022ANP}}. {We assume that the solution to the MANP equations \eqref{MANP} with periodic boundary conditions exists} and has several physical properties of importance, such as solution positivity, mass conservation, and energy dissipation:
\begin{align}
&{ c^\ell (\boldsymbol{r},t) > 0 \text{~and~} c^0(\boldsymbol{r},t) >0},  \quad \boldsymbol{r}\in \Omega ,t>0,  \label{pos} \\
&\int _{\Omega} c^\ell (\boldsymbol{r},t) d\boldsymbol{r}=\int _{\Omega} c^\ell(\boldsymbol{r},0) d\boldsymbol{r}, \label{cons}\\
&\frac{d\mathcal{F}}{dt} = - \sum_{\ell=1}^M \int_{\Omega} \kappa c^\ell \left| \nabla \log c^\ell - q^\ell \frac{\boldsymbol{D}}{\ve}+\nabla \mu^{\ell,\text{cr}}  \right|^2 d\boldsymbol{r} \leq 0.  \label{disspt}
\end{align}
{Here the electrostatic free energy is defined as a functional of $c^\ell$:
\begin{equation}\label{EDc}
\mathcal{F}[c^1, \dots, c^M] = \int_{\Omega} \left(\kappa^2\frac{|\bm{D}|^2}{\varepsilon} + \sum_{\ell=1}^{M} c^\ell\left(\log c^\ell+\mu^{\ell,\text{cr}}\right)  \right) d\bm{r}, \text{ with } \nabla \cdot 2\kappa^2 \bm{D}=\rho.
\end{equation}
It is obvious that $\mathcal{F}[ c^1, \dots, c^M]$ is a convex energy functional with respect to $c^{\ell}$.}  {The positivity of $c^0$ is assumed in \reff{pos}, and the theoretical analysis on the positivity in equilibrium is referred to~\cite{citeone}.}

By the energy dissipation law~\reff{disspt}, the steady state of the MANP model reproduces the charge conserving Poisson--Boltzmann (ccPB) theory~\cite{Lee_JMP2014,Lee_NonL2010, Wan_PRX2014}
\begin{equation}
	-\nabla \cdot 2\kappa^2\varepsilon \nabla \phi=\sum_{\ell=1}^{M} \frac{q^\ell N^{\ell}}{\int_{\Omega} e^{-q^\ell \phi} d\bm{r}}  e^{-q^\ell \phi}  + \rho ^f,
\end{equation}
where $N^{\ell}$ is the total ionic mass of the $\ell$-th species. In practical applications, it is desirable and of great importance to design  numerical methods that are able to preserve the intrinsic properties~\reff{pos}-\reff{disspt} at the discrete level.

\section{Numerical method}\label{s:NumMethod}
\subsection{Discretization of Maxwell-Amp\`{e}re Nernst-Planck equations}
Let $\Delta t$ be the time step size. Denote by $c^{\ell,n}$ and $\bm{D}^{n}$ the approximations to $c^\ell$ and $\bm{D}$, respectively, at time $t_n:=n\Delta t$ for a nonnegative integer $n$. An Euler discretization of the time derivative leads to a semi-implicit scheme for the Nernst-Planck equations
\begin{equation}
\frac{c^{\ell,n+1}-c^{\ell,n}}{\Delta t}= -\nabla \cdot \bm{J}^{\ell, n} , \label{disnp}
\end{equation}
where
\[
\bm{J}^{\ell, n} = -\kappa\left(\nabla c^{\ell,n+1}-\frac{q^\ell c^{\ell,n+1} \bm{D}^n}{\ve}+c^{\ell,n+1}\nabla \mu^{\ell,\text{cr}}\right).
\]
Similarly, the Maxwell-Amp\`{e}re equation is discretized by
\begin{equation}\label{disam}
\frac{\bm{D}^{*}-\bm{D}^{n}}{\Delta t} =-\sum _{\ell=1}^M \frac{q^\ell \bm{J}^{\ell, n}}{2\kappa^2 } +\bm{\Theta}^n,
\end{equation}
where $\bm{D}^{*}$ is an interim approximation to $\bm{D}^{n+1}$. It follows from~\reff{disnp} and \reff{disam} that $\bm{D}^{*}$ satisfies the Gauss's law for the updated concentrations $c^{\ell,n+1}$, provided that $\bm{D}^{n}$ satisfies the Gauss's law. In the next step, $\bm{D}^{*}$ will be further corrected in a Gauss-law-satisfying manner to fulfill the curl-free constraint in~\reff{MANP}.
In general, $\bm{\Theta}^n$ can be arbitrarily chosen as long as it satisfies the divergence free condition. However,
an appropriate estimate of $\bm{\Theta}^n$ can lead to a $\bm{D}^*$ that is closer to the curl-free condition, and reduce the number of relaxation steps in the correction step from $\bm{D}^{*}$ to $\bm{D}^{n+1}$. {Extensive numerical tests have demonstrated that with
\begin{equation}\label{Theta2}
	\bm{\Theta}^n=\frac{\bm{D}^{n} - \bm{D}^{n-1}}{{\Delta t}}+\sum _{\ell=1}^M \frac{q^\ell \bm{J}^{\ell, n-1}}{2\kappa^2},
\end{equation}
the number of relaxation steps is small, as low as one step~\cite{Qiao2022ANP}. Thus, \reff{Theta2} shall be considered in our numerical simulations.}


For simplicity of presentation, we consider spatial discretization in a rectangular domain $\Omega=[0, L_1]\times [0,L_2]$ with periodic boundary conditions. The domain is covered by a uniform grid  with the grid spacing $\Delta x$ and $\Delta y$ in the two coordinate directions. Denote by $\Delta \Omega =\Delta x \Delta y$ the area of a unit cell. Let $N_x$ and $N_y$ be the number of grid points along each dimension. Let $\bm{D}=(D_x, D_y)$. $c_{i,j}^{\ell,n}$ stands for the numerical approximation of $c^{\ell,n}$ on the grid point $\left(i\Delta x, j\Delta y\right)$, and $D_{i+1/2, j}^{n}(\text{resp. } D_{i, j+1/2}^{n})$ stands for the numerical approximation of $D_x^{n}\left(\text{resp. }D_y^n\right)$ on the grid point $\left((i+1/2)\Delta x , j \Delta y\right)\left(\text{resp. }\left(i\Delta x , (j+1/2) \Delta y\right)\right)$ for  $i=1, \ldots, N_{x}$ and  $j=1, \ldots, N_{y}$.

By the curl-free condition in \reff{MANP}, there exists a scalar function $\phi$ such that $ \bm{D}/\varepsilon=-\nabla \phi$. Using $\phi$, the semi-discrete Nernst-Planck equations~\reff{disnp} with the Slotboom transformation~\cite{Ding19JCP, Ding2020JCP, LM2020,LW:JCP:14} can be rewritten as
\begin{equation}\label{Slotboom}
	\frac{c^{\ell,n+1}-c^{\ell,n}}{\Delta t}=\nabla \cdot \kappa \left[ e^{-g^{\ell,n} }\nabla (e^{g^{\ell,n}}c^{\ell,n+1}) \right],
\end{equation}
where $g^{\ell,n}= q^\ell \phi^n+\mu^{\ell,\text{cr}, n}$. Central differencing approximation leads to
\begin{equation}
	\frac{c^{\ell,n+1}_{i,j}-c^{\ell,n}_{i,j}}{\Delta t}=-\frac{J^{\ell,n}_{i+\frac{1}{2},j}-J^{\ell,n}_{i-\frac{1}{2},j}}{\Delta x}-\frac{J^{\ell,n}_{i,j+\frac{1}{2}}-J^{\ell,n}_{i,j-\frac{1}{2}}}{\Delta y},\label{disc}
\end{equation}
where the numerical fluxes are given by
$$
J_{i+\frac{1}{2},j}^{\ell,n}=-\kappa  e^{-g^{\ell,n}_{i+\frac{1}{2},j}}\frac{e^{g^{\ell,n}_{i+1,j}}c^{\ell,n+1}_{i+1,j}-e^{g^{\ell,n}_{i,j}}c^{\ell,n+1}_{i,j}}{\Delta x}
$$
and
$$
J_{i,j+\frac{1}{2}}^{\ell,n}=-\kappa e^{-g^{\ell,n}_{i,j+\frac{1}{2}}}\frac{e^{g^{\ell,n}_{i,j+1}}c^{\ell,n+1}_{i,j+1}-e^{g^{\ell,n}_{i,j}}c^{\ell,n+1}_{i,j}}{\Delta y}.
$$
Here the values at half grid points are approximated by the entropic mean
\begin{equation}
	e^{-g^{\ell,n}_{i+\frac{1}{2},j}}=\frac{g^{\ell,n}_{i+1,j}-g^{\ell,n}_{i,j}}
	{e^{g^{\ell,n}_{i+1,j}}-e^{g^{\ell,n}_{i,j}}}
	~~\text{and}~~
	e^{-g^{\ell,n}_{i,j+\frac{1}{2}}}=\frac{g^{\ell,n}_{i,j+1}-g^{\ell,n}_{i,j}}
	{e^{g^{\ell,n}_{i,j+1}}-e^{g^{\ell,n}_{i,j}}}.\label{har}
\end{equation}
\begin{rem}\label{Means}
There are other approximations of $e^{-g^{\ell,n}_{i+1/2,j}}$ that can be used to derive different numerical schemes for the NP equations.
\begin{compactenum}
\item[\rm (1)]
By the arithmetic mean $$e^{-g^{\ell,n}_{i+1/2,j}}=\frac{e^{-g^{\ell,n}_{i+1,j}}+e^{-g^{\ell,n}_{i,j}}}{2},$$ the function $B(\cdot)$ in the numerical flux~\reff{J} reads $B(z)=(1+e^{-z})/2$.
\item[\rm (2)]
By the geometric mean $$e^{-g^{\ell,n}_{i+1/2,j}}=e^{-\frac{g^{\ell,n}_{i+1,j}+g^{\ell,n}_{i,j}}{2}},$$
the function $B(\cdot)$ in the numerical flux~\reff{J} reads $B(z)=e^{-z/2}$.
\item[\rm (3)]
By the harmonic mean approximation
$$e^{-g^{\ell,n}_{i+1/2,j}}=\frac{2 e^{-g^{\ell,n}_{i+1,j}-g^{\ell,n}_{i,j}}}{e^{-g^{\ell,n}_{i+1,j}}+e^{-g^{\ell,n}_{i,j}}},$$
the function $B(\cdot)$ in the numerical flux~\reff{J} reads $B(z)=2/\left(1+e^z\right)$, which has been shown to result in a linear system with a well-conditioned coefficient matrix~\cite{Ding19JCP,QianWangZhou_JCP19}.
One can find that the coefficient in the stencil is naturally bounded, which plays a pivotal role in the estimate on condition numbers of the coefficient matrix of a semi-implicit scheme.
\end{compactenum}
\end{rem}

By the central differencing
\begin{equation}
		\frac{D_{i+\frac{1}{2},j}}{\varepsilon_{i+\frac{1}{2},j}}\approx-\frac{\phi_{i+1,j}-\phi_{i,j}}{\Delta x}, ~~\frac{D_{i,j+\frac{1}{2}}}{\varepsilon_{i,j+\frac{1}{2}}}\approx-\frac{\phi_{i,j+1}-\phi_{i,j}}{\Delta y}\label{dphi},
\end{equation}
we have
\begin{equation}\label{J}
J_{i+\frac{1}{2},j}^{\ell,n}=-\frac{\kappa}{\Delta x} \left[
B\left(-dg^{\ell,n}_{i+\frac{1}{2},j}\right)c^{\ell,n+1}_{i+1,j}- B\left(dg^{\ell,n}_{i+\frac{1}{2},j}\right)c^{\ell,n+1}_{i,j} \right],
\end{equation}
where $B(\cdot)$ is the Bernoulli function defined by
\[
B(z) = \frac{z}{e^z -1}~\text{for}~ z\neq0, ~\text{and} ~ B(0)=1,
\]
and
$$dg_{i+\frac{1}{2},j}^{\ell,n}=g^{\ell,n}_{i+1,j}-g^{\ell,n}_{i,j}=-\Delta x q^{\ell}\frac{D_{i+1/2,j}^n}{\ve_{i+1/2,j}}+\mu^{\ell,\text{cr}, n}_{i+1,j}-\mu^{\ell,\text{cr},n }_{i,j}.$$
The numerical fluxes $J_{i-\hf,j}^{\ell,n}$, $J_{i,j+\hf}^{\ell,n}$, and $J_{i,j-\hf}^{\ell,n}$ can be given analogously. {We assume that $c^{0,n}_{i,j}$ on grid points are all positive, so that $\mu^{\ell,\text{cr},n }_{i,j}$ for $\ell = 1, \cdots, M$ are well defined. It is non-trivial to derive a sufficient condition to guarantee the positivity of $c^{0,n}_{i,j}$, which is explicitly treated in our numerical scheme. We leave the numerical analysis on its positivity to our future work. }

{As mentioned in Remark~\ref{Means}, other approximations of $e^{-g^{\ell,n}_{i+1/2,j}}$ at half points can be employed in spatial discretization. The entropic mean presented in \eqref{har} results in the well-known Scharfetter-Gummel scheme~\cite{Farrel1991SG,Scharfetter1969IEEE}, which can automatically reduce to an upwinding scheme when the convection arising from electric fields, Born solvation and steric effects is large.}

\begin{rem}\label{ModifiedSG}
Modified Scharfetter-Gummel fluxes can be derived by looking for a Lipschitz continuous function $B: \R \to \R^+$ in the numerical flux~\reff{J} satisfying conditions:
\[
\lim_{z\to0} B(z) =1,~ \lim_{z\to +\infty} B(z) =0,  \mbox{~{\rm and}} \lim_{z\to -\infty} B(z)/z =-1.
\]
Such modified Scharfetter-Gummel fluxes, which automatically reduce to upwinding fluxes for convection-dominated cases, can be used in other spatial discretization methods.
\end{rem}

With the numerical fluxes~\reff{J}, the fully discrete Maxwell-Amp\`{e}re equation reads
\begin{align}
	\frac{D^{*}_{i+\frac{1}{2},j}-D^{n}_{i+\frac{1}{2},j}}{\Delta t} &=- \sum _\ell^M \frac{q^\ell J^{\ell,n}_{i+\frac{1}{2},j}}{2\kappa^2} +\Theta^n_{i+\frac{1}{2},j}, \label{D*1}\\
	\frac{D^{*}_{i,j+\frac{1}{2}}-D^{n}_{i,j+\frac{1}{2}}}{\Delta t} &=- \sum _\ell^M \frac{q^\ell J^{\ell,n}_{i,j+\frac{1}{2}}}{2\kappa^2} +\Theta^n _{i,j+\frac{1}{2}}. \label{D*2}
\end{align}
Notice that the numerical schemes~\reff{D*1} and \reff{D*2} for the Maxwell-Amp\`{e}re equation are explicit after $c^{\ell,n+1}$ is obtained.
\begin{rem}
Alternatively, one can employ second-order temporal discretization:
\begin{align}
	&\frac{3c^{\ell,n+1}_{i,j}-4c^{\ell,n}_{i,j}+c^{\ell,n-1}_{i,j}}{2\Delta t}=
	-\frac{J^{\ell,n+1}_{i+\frac{1}{2},j}-J^{\ell,n+1}_{i-\frac{1}{2},j}}{\Delta x}-\frac{J^{\ell,n+1}_{i,j+\frac{1}{2}}-J^{\ell,n+1}_{i,j-\frac{1}{2}}}{\Delta y},\label{cn}\\
	&\frac{3D^{*}_{i+\frac{1}{2},j}-4D^{n}_{i+\frac{1}{2},j}+D^{n-1}_{i+\frac{1}{2},j}}{2\Delta t} =- \sum_\ell^M \frac{q^\ell J^{\ell,n+1}_{i+\frac{1}{2},j}}{2\kappa^2} +\Theta^{n+1}_{i+\frac{1}{2},j}, \label{D*1cn}\\
	&\frac{3D^{*}_{i,j+\frac{1}{2}}-4D^{n}_{i,j+\frac{1}{2}}+D^{n-1}_{i,j+\frac{1}{2}}}{2\Delta t} =- \sum_\ell^M \frac{q^\ell J^{\ell,n+1}_{i,j+\frac{1}{2}}}{2\kappa^2} +\Theta^{n+1} _{i,j+\frac{1}{2}}, \label{D*2cn}
\end{align}
where the numerical fluxes
$$
J_{i+\frac{1}{2},j}^{\ell,n+1}=-\frac{\kappa}{\Delta x} \left[
B\left(-dg^{\ell,n+1}_{i+\frac{1}{2},j}\right)c^{\ell,n+1}_{i+1,j}- B\left(dg^{\ell,n+1}_{i+\frac{1}{2},j}\right)c^{\ell,n+1}_{i,j} \right],
$$
and
$$
\Theta^{n+1}_{i+\frac{1}{2},j}=2\Theta^{n}_{i+\frac{1}{2},j}-\Theta^{n-1}_{i+\frac{1}{2},j}.$$
Here,
$$dg^{\ell,n+1}_{i+\frac{1}{2},j}=2dg^{\ell,n}_{i+\frac{1}{2},j}-dg^{\ell,n-1}_{i+\frac{1}{2},j}=2\left(g^{\ell,n}_{i+1,j}-g^{\ell,n}_{i,j}\right)-\left(g^{\ell,n-1}_{i+1,j}-g^{\ell,n-1}_{i,j}\right).$$
The numerical fluxes $J_{i-1/2,j}^{\ell,n+1}$, $J_{i,j+1/2}^{\ell,n+1}$, $J_{i,j-1/2}^{\ell,n+1}$ and $\Theta_{i,j+1/2}^{n+1}$ can be given analogously.
\end{rem}

\subsection{Local curl-free algorithm}\label{localsec}
The interim electric displacement $\bm{D}^{*}$ updated in \reff{disam} often fails to satisfy the curl-free constraint in the MANP formulation~\reff{MANP}. A Gauss-law satisfying relaxation algorithm~\cite{BSD:PRE:2009, MR:PRL:2002,RM:JCP:2004MC,  RM:PRL:2004, ZWL:PRE:2011} has been proposed to achieve the curl-free constraint, starting from $\bm{D}^{*}$. The relaxation algorithm has been incorporated to numerically solve the MANP model in the work~\cite{Qiao2022ANP}.  Here we recall the main steps with attention on how to couple the relaxation algorithm with the proposed numerical scheme for the NP equations.

Following the idea originally proposed in~\cite{MR:PRL:2002}, one considers a convex constraint optimization problem
\begin{equation}\label{MinPro}
\min \limits_{\bm{D}}~\mathcal{F}_{\rm pot}[\bm{D}] :=  \int_{\Omega} \frac{\kappa^2|\bm{D}|^2}{\varepsilon} d\bm{r}, ~~\text{s.t.}~ \nabla \cdot 2\kappa^2 \bm{D}=\rho.
\end{equation}
Using the Lagrange multiplier method, it is easy to find that there exists the unique minimizer $\bm{D}=-\ve \nabla \phi$, where $\phi$ is a Lagrange multiplier satisfying the Poisson's equation $-\nabla \cdot 2\kappa^2\ve \nabla \phi = \rho$. Obviously, the minimizer $\bm{D}$ satisfies the curl-free constraint $\nabla \times (\bm{D}/\ve) = 0$.  This observation offers an optimization approach starting from $\bm{D}^{*}$ to reach the curl-free constraint, while still retaining the {discrete} Gauss's law.

\begin{figure}[h!]
	\centering
	\includegraphics[scale=0.8]{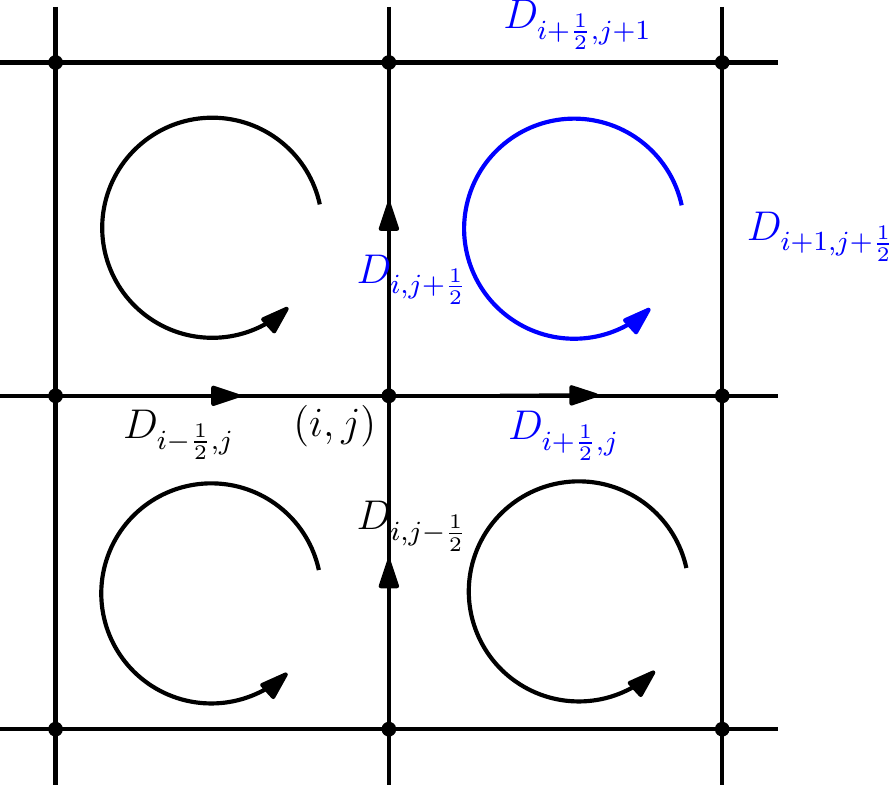}
	\caption{Update diagram of the electric displacements associated to a node $(i,j)$.}
	\label{fD}
\end{figure}

With central differencing, the electric energy in the optimization problem~\reff{MinPro} is approximated with second-order accuracy by
\begin{equation}\label{Fpoth}
	\mathcal{F}_{\rm pot}^{h}=\Delta {\Omega} \kappa^2 \sum_ {i,j} \left( \frac{D^2_{i+\frac{1}{2},j}}{\varepsilon_{i+\frac{1}{2},j}}  + \frac{D^2_{i,j+\frac{1}{2}}}{\varepsilon_{i,j+\frac{1}{2}}} \right),
\end{equation}
and the Gauss's law is approximated by
\begin{equation}
	2\kappa^2\frac{D_{i+\frac{1}{2},j}-D_{i-\frac{1}{2},j}}{\Delta x}
	+2\kappa^2\frac{D_{i,j+\frac{1}{2}}-D_{i,j-\frac{1}{2}}}{\Delta y}
	=\sum_\ell q^\ell c^\ell_{i,j} +\rho^f _{i,j}.\label{gausslaw}
\end{equation}
This relaxation starts from the interim approximation, $\bm{D}^*$, {which satisfies the discrete Gauss's law~\reff{gausslaw}~\cite{Qiao2022ANP}}, and performs successive local updates of the electric displacements in each single cell to minimize $\mathcal{F}_{\rm pot}^{h}$~\cite{MR:PRL:2002}. For instance, consider a grid cell that has four nodes with indices $(i,j)$, $(i+1,j)$, $(i, j+1)$, and $(i+1, j+1)$, and four corresponding edges; cf. Fig.~\ref{fD}. Denote by $D_{i+1/2, j}$, $D_{i+1, j+1/2}$, $D_{i+1/2, j+1}$, and $D_{i, j+1/2}$ the four electric displacement components defined on the four edges.  Introducing an update flux $\eta$, one can update the displacements by
\begin{equation}
\begin{aligned}
	D_{i+\frac12, j} &\leftarrow D_{i+\frac12, j}+\frac{\eta}{\Delta y} ,\\
	D_{i+1, j+\frac12} &\leftarrow D_{i+1, j+\frac12}+\frac{\eta}{\Delta x},\\
	D_{i+\frac12, j+1} &\leftarrow D_{i+\frac12, j+1}-\frac{\eta}{\Delta y} ,\\
	D_{i, j+\frac12}     &\leftarrow D_{i, j+\frac12}-\frac{\eta}{\Delta x}.
\end{aligned}\label{updateD}
\end{equation}
As shown in Fig.~\ref{fD}, the {discrete} Gauss's law is rigorously maintained at all the four nodes after four circulatory updates in that the sum of entering and leaving fluxes are unchanged at each node during each update.  In each update, the optimal flux is determined by the minimization of the discrete energy $\mathcal{F}^h_{\text{pot}}$ locally and it can be calculated explicitly:
\begin{equation}\label{eta}
		\eta= - \frac{\Delta y(\Delta x)^2\left(\frac{D_{i+\frac{1}{2},j}}{\varepsilon_{i+\frac{1}{2},j}}-\frac{D_{i+\frac{1}{2},j+1}}{\varepsilon_{i+\frac{1}{2},j+1}}\right)+  \Delta x(\Delta y)^2\left(\frac{D_{i+1,j+\frac{1}{2}}}{\varepsilon_{i+1,j+\frac{1}{2}}}-\frac{D_{i,j+\frac{1}{2}}}{\varepsilon_{i,j+\frac{1}{2}}}\right)}{(\Delta x)^2 \left(\frac{1}{\varepsilon_{i+\frac{1}{2},j}}+\frac{1}{\varepsilon_{i+\frac{1}{2},j+1}} \right) +(\Delta y)^2 \left(\frac{1}{\varepsilon_{i,j+\frac{1}{2}}}+\frac{1}{\varepsilon_{i+1,j+\frac{1}{2}}} \right)}.
\end{equation}
Such a local update loops over all the cells for certain steps until a stopping criterion $\ve_{\rm tol}$ is met. From the description of the algorithm, analytical expressions are available for each update step. Therefore, the complexity of the algorithm is linear when the number of looping steps is bounded. This local update has robust convergence and stability by the fact that $\mathcal{F}_{\rm pot}^{h}$ is nonnegative and $\delta \mathcal{F}_{\rm pot}^{h}(\eta)\leq 0 $ when $\eta$ is given by~\reff{eta}. {Also, it can be proven that the curl-free constraint is achieved as the local update algorithm converges~\cite{Qiao2022ANP}.} The algorithm is also directly applicable to three dimensional cases by circulatorily updating electric displacements on each face of cubic cells.

In summary, the whole numerical algorithm for the MANP equations is presented Algorithm 1.

\begin{algorithm}[H] \label{A:alg}
	\caption{Numerical method for the MANP equations}
	{\bf Input:}  Final time $T$, stopping criterion $\ve_{\rm tol}$ for the local curl-free algorithm, initial ionic distributions $c^{\ell, 0}$, and the corresponding displacement field $\bm{D}^0$ that satisfies the {discrete} Gauss's law;
	\begin{algorithmic}[1]
 		\State Solve the semi-implicit scheme~\reff{disc} for the NP equations to get $c^{\ell,n+1}$, $\ell=1, \dots, M$;
		 \State Given a divergence free $\bm{\Theta}^n$ in \reff{Theta2} and obtained $c^{\ell,n+1}$, explicitly solve the scheme~\reff{D*1}-\reff{D*2} for the Maxwell-Amp\`{e}re equation to obtain an interim displacement field $\bm{D}^*$;
 		\State Perform the local curl-free algorithm for $\bm{D}^*$ to obtain $\bm{D}^{n+1}$.
	\State If $T\leq t^{n+1}$, stops; Otherwise, go to Step 1.
	\end{algorithmic}
\end{algorithm}

\begin{rem}
Note that Algorithm $1$ can be readily extended to three dimensions in a dimension-by-dimension manner.
\end{rem}


\subsection{Analysis on structure-preserving properties}\label{s:TheAna}
The proposed numerical scheme~\reff{disc} and \reff{D*1}-\reff{D*2}, combined with the local curl-free algorithm in Section~\ref{localsec}, is able to preserve mass conservation, solution positivity, and energy dissipation of the MANP equations at the discrete level.

\begin{thm}\label{thm1}
	(Mass conservation)
	The numerical scheme~\reff{disc} for the NP equations with periodic boundary conditions satisfies the conservation law of mass at the discrete level, in the sense that  the total concentration remains constant in time, i.e.,
	\begin{equation}
		\Delta\Omega \sum_{i=1}^{N_x} \sum_{j=1}^{N_y}   c_{i,j}^{\ell, n+1} =\Delta \Omega \sum_{i=1}^{N_x}  \sum_{j=1}^{N_y}  c_{i,j}^{\ell, n} .
	\end{equation}
\end{thm}

\begin{proof}
	Summing both sides of~\reff{disc} leads to
\[
		\sum_{i,j} \Delta \Omega \left( c^{\ell,n+1}_{i,j}-c^{\ell,n}_{i,j} \right) = -\Delta t \sum_{i,j} \Delta \Omega \left( \frac{J^{\ell,n}_{i+\frac{1}{2},j} - J^{\ell,n}_{i-\frac{1}{2},j} }{\Delta x} + \frac{J^{\ell,n}_{i,j+\frac{1}{2}} - J^{\ell,n}_{i,j-\frac{1}{2}} }{\Delta y} \right)=0,
\]
	where the summation by parts and periodic boundary conditions have been used in the second equality.
\end{proof}

\begin{thm}\label{thm2}
	(Positivity preservation) The numerical scheme~\reff{disc} for the NP equations preserves the solution positivity in time,  i.e.,
	$$
	c_{i,j}^{\ell, n+1}>0 \text { for } i=1, \cdots, N_x,  j=1,  \cdots, N_y,
	$$
provided that $c_{i,j}^{\ell, n}>0.$
\end{thm}

\begin{proof}
The fully discrete numerical scheme for the NP equations can be rewritten in a matrix form
\begin{equation}\label{linearc}
	\mathcal{L} c^{\ell,n+1}=c^{\ell,n},
\end{equation}
where $\mathcal{L}$ is a coefficient matrix involving $\bm{D}^n$ and $\mu^{\ell,\text{cr}, n}$.
    For simplicity of presentation, we introduce an index conversion notation
    $$
    [i, j]:=(i-1) N_{y}+j \quad \text {for } i=1, \ldots, N_{x},~ j=1, \ldots, N_{y},
    $$
    indices for periodic boundary conditions:
    $$
    i^{+}=\left\{\begin{array}{l}
    	i+1~ \text{ for } i=1,\ldots , N_x-1,\\
    	1~~~~~ \text{ for }  i=N_x,
    \end{array}\right. \qquad
     i^{-}=\left\{\begin{array}{l}
    	i-1~ \text{ for } i=2,\ldots , N_x,\\
    	N_x~~~~ \text{ for }  i=1,
    \end{array}\right.
    $$

    $$
    j^{+}=\left\{\begin{array}{l}
    	j+1~ \text{ for } j=1,\ldots , N_y-1,\\
    	1~~~~~ \text{ for }  j=N_y,
    \end{array}\right. \qquad
     j^{-}=\left\{\begin{array}{l}
    	j-1~ \text{ for } j=2,\ldots , N_y,\\
    	N_y~~~~ \text{ for }  j=1,
    \end{array}\right.
    $$
    and discrete operators
    $$
    \mathcal{H}^{\ell, \pm}_{x} D^n_{i,j}=\kappa B\left( \pm dg_{i\pm 1/2,j}^{\ell,n} \right) \quad \text{and} \quad
    \mathcal{H}^{\ell, \pm}_{y} D^n_{i,j}=\kappa B\left( \pm dg_{i,j\pm 1/2}^{\ell,n} \right).
    $$
We study the elements of $\mathcal{L}$ by looking into each column. The non-zero entries of the $k$th $(k=[i,j])$ column of $\mathcal{L}$ are given by
	\begin{equation}
	\mathcal{L}_{m, k}=\left\{
		\begin{aligned}
			&-\frac{\Delta t}{\Delta x^2} \mathcal{H}^{\ell, -}_{x} D^n_{i,j}, \quad m=[i^{-}, j], \\
			&-\frac{\Delta t}{\Delta y^2} \mathcal{H}^{\ell, -}_{y} D^n_{i,j}, \quad m=[i, j^{-}], \\
			&1+\left[\frac{\Delta t}{\Delta x^2} \left( \mathcal{H}^{\ell, +}_{x} +\mathcal{H}^{\ell, -}_{x}  \right)
			+\frac{\Delta t}{\Delta y^2} \left( \mathcal{H}^{\ell, +}_{y} +\mathcal{H}^{\ell, -}_{y}  \right)\right] D^n_{i,j},~ m=k, \\
			&-\frac{\Delta t}{\Delta y^2} \mathcal{H}^{\ell, +}_{y} D^n_{i,j}, \quad m=[i, j^{+}], \\
			&-\frac{\Delta t}{\Delta x^2}\mathcal{H}^{\ell, +}_{x} D^n_{i,j}, \quad m=[i^{+}, j].
			\end{aligned}
		\right.\label{matrixl}
	\end{equation}
It is easy to verify that $\mathcal{L}$ has positive diagonal elements and negative off-diagonal elements, and is strictly diagonally dominant according to the column sum $\sum_{m=1}^{N_xN_y} \mathcal{L}_{m, k} =1$.  As a result, $\mathcal{L}$ is an M-matrix with $\mathcal{L}^{-1}\succ 0$. Thus, if $c_{i,j}^{\ell, n}>0,$ then
$$
c_{i,j}^{\ell, n+1}>0, \text { for } i=1, \cdots, N_x, ~j=1, \cdots, N_y.
$$
\end{proof}

Since energy dissipation plays a crucial role in both theoretical analysis and numerical stability, it is highly desirable to maintain this property at the discrete level in numerical simulations. With a second-order accurate quadrature, the electrostatic free energy $\mathcal{F}$ in~\reff{EDc} can be approximated by
\begin{equation}
	\mathcal{F}_h^n=\Delta \Omega \sum_{i=1}^{N_x} \sum_{j=1}^{N_y} \left( \frac{\kappa^2\left|D_{i+\frac{1}{2},j}^n\right|^2}{\varepsilon_{i+\frac{1}{2},j}}+\frac{\kappa^2\left|D_{i,j+\frac{1}{2}}^n\right|^2}{\varepsilon_{i,j+\frac{1}{2}}} \right)
	+ \Delta \Omega \sum_{\ell=1}^M \sum_{i=1}^{N_x} \sum_{j=1}^{N_y}  c^{\ell,n}_{i,j} \left(\log(c^{\ell,n}_{i,j})+\mu^{\ell,\text{cr},n}_{i,j} \right). \label{energydis}
\end{equation}

\begin{thm}\label{thm4} (Energy dissipation)  Let $\varepsilon_{\min}=\min \limits_{i,j} \left\{ \varepsilon_{i+1/2, j}, \varepsilon_{i, j+1/2}\right\}$, $\varepsilon_{\max }=\max \limits_{i,j} \left\{ \varepsilon_{i+1/2,j},  \varepsilon_{i,j+1/2}\right\}$,  and $c_{\max}=\max \limits_{i,j,\ell} \left\{c_{i,j}^{\ell,n+1}\right\}$. Define
\begin{equation}\label{time_cons}
	\Delta t^{*}=\frac{ 2 \kappa\varepsilon_{\min }^{3}}{\varepsilon_{\max }^2  c_{\max} \sum_{\ell=1}^{M} \left|q^\ell \right|^{2}} {\rm exp}\left[-  \max _{i,j,\ell} \left(\left| dg^{\ell,n}_{i+\frac{1}{2},j} \right|,\left| dg^{\ell,n}_{i,j+\frac{1}{2}} \right| \right) \right].
\end{equation}
If $\mu^{\ell,\text{cr}}$ is time independent and $\Delta t \in\left(0, \Delta t^{*}\right)$, then the numerical scheme~\reff{disc} and \reff{D*1}-\reff{D*2} satisfies the discrete energy dissipation
	\begin{equation}
	\mathcal{F}_{h}^{n+1}-\mathcal{F}_{h}^{n} \leq-\frac{\Delta t}{2} I_1, \label{energyinequ}
	\end{equation}
	where
	$$
	\begin{aligned}
	I_1=- \sum_{\ell=1}^M \sum_{i,j} \Delta \Omega &\left[ J_{i+\frac{1}{2},j}^{\ell,n} \left( \frac{\log \left(c_{i+1,j}^{\ell,n+1} \right)-\log \left(c_{i,j}^{\ell,n+1} \right)+\mu^{\ell,\text{cr}}_{i+1,j}-\mu^{\ell,\text{cr}}_{i,j}}{\Delta x} - \frac{q^\ell D^n_{i+\frac{1}{2},j}}{\varepsilon_{i+\frac{1}{2},j}} \right) \right.\\
	&\left. + J_{i,j+\frac{1}{2}}^{\ell,n} \left( \frac{\log \left(c_{i,j+1}^{\ell,n+1} \right)-\log \left(c_{i,j}^{\ell,n+1} \right)+\mu^{\ell,\text{cr}}_{i,j+1}-\mu^{\ell,\text{cr}}_{i,j}}{\Delta y} - \frac{q^\ell D^n_{i,j+\frac{1}{2}}}{\varepsilon_{i,j+\frac{1}{2}}} \right) \right] \geq 0.
    \end{aligned}
	$$
\end{thm}
The theorem can be proved by following the idea presented in the work~\cite{LM2020}, with modifications on the treatment of the electric displacement, in place of the electric potential, due to the local relaxation update and entropic mean approximation in the numerical fluxes. See the proof given in the Appendix.

\begin{rem}
The condition \eqref{time_cons} implies an assumption on the uniform discrete $L^\infty$ boundness of $c_{i,j}^{\ell,n+1}$. Although this assumption can be validated by the convergence analysis based on the higher order asymptotic expansion~\cite{Ding2022con,LiuChun2020positivity},  whether they could be applied on the MANP model is still open. 
Fortunately, in practical simulations, the computed solutions are always bounded, so our theoretical results are still valid in this version.
\end{rem}

\section{Numerical Results}\label{s:NumRes}
To demonstrate the effectiveness of the proposed scheme, we conduct a series of numerical tests on accuracy, as well as their performance in preserving solution positivity, mass conservation, and energy dissipation at the discrete level.
\subsection{Accuracy of the algorithm}\label{accuracy}
In this example, we consider charge dynamics of binary electrolytes in a rescaled domain $\Omega=[-1, 1]\times [-1, 1]$ with periodic boundary conditions.  Simulations are performed by numerically solving the MANP equations
\[
	\left\{
	\begin{aligned}
		&\partial_{t} c^\ell=-\nabla \cdot \bm{J}^\ell, ~\ell=1,2, \\
		&\bm{J}^\ell=- \kappa\left( \nabla c^\ell-q^\ell c^\ell \bm{D}/\ve+\bm{g}^\ell \right), ~\ell=1,2,\\
		&\partial_{t} \bm{D}=\left(-\bm{J}^1+\bm{J}^2+\bm{S} \right)/\left( 2\kappa^2\right)+\bm{\Theta},\\
		&\nabla \times \bm{D}/\ve=0, \\
		&\nabla \cdot \bm{\Theta}=0,
	\end{aligned}
	\right.
\]
with the dielectric coefficient $\varepsilon=0.5$, valences $q^\ell=(-1)^{\ell+1}$, and $\kappa=1$.
The source terms, as well as initial conditions, are determined by the following exact solution
\[
	\left\{
	\begin{aligned}
	&c^\ell(x, y, t)=\pi^{2} e^{-t} \cos (\pi x) \cos (\pi y) / 5+2, ~\ell=1,2, \\
	&\bm{D}(x, y)=\left(\begin{aligned}\pi e^{-t} \sin (\pi x) \cos (\pi y)/2
		\\ \pi e^{-t} \cos (\pi x) \sin (\pi y)/2
				\end{aligned}\right).
	\end{aligned}
	\right.
\]
Note that $\bm{g}^\ell$ and $\bm{S}$ are not unique and only one specific set of periodic solutions is needed. We choose
\[
	\left\{
	\begin{aligned}
        & \bm{g}^1=\left( \begin{array}{c}
        	     (\frac{2\pi^3}{5}+4\pi-\frac{\pi}{5})e^{-t} \sin(\pi x)\cos(\pi y) +\frac{\pi^3}{10} e^{-2t} \sin(2\pi x)\cos^2(\pi y)\\
        	     \frac{\pi^3}{10}e^{-2t}\cos^2(\pi x) \sin(2\pi y)
        				\end{array}
				\right),\\
	& \bm{g}^2=\left( \begin{array}{c}
		(\frac{2\pi^3}{5}-4\pi-\frac{\pi}{5})e^{-t} \sin(\pi x)\cos(\pi y) -\frac{\pi^3}{10} e^{-2t} \sin(2\pi x)\cos^2(\pi y)\\
		-\frac{\pi^3}{10}e^{-2t}\cos^2(\pi x) \sin(2\pi y)
				\end{array}
				\right),\\
	& \bm{S}=\left( \begin{array}{c}
	    	-2\pi e^{-t} \sin(\pi x)\cos(\pi y)\\
			0
			\end{array}
			\right).
	\end{aligned}
	\right.
\]

To verify the accuracy of the semi-implicit scheme~\reff{disc} and \reff{D*1}-\reff{D*2}, we perform numerical simulations on a uniform mesh with grid sizes $\Delta x=\Delta y=h$. Numerical errors on various refined meshes are obtained by comparing the numerical solutions against the exact solution.

\begin{table}[H]
	\centering
	\setlength{\abovecaptionskip}{10pt}%
	\setlength{\belowcaptionskip}{0pt}%
	\begin{tabular}{ccccc}	
		\hline \hline
		$h$  & $c^1$  & Order  & $c^2$    & Order  \\
		\hline
		0.1   & 1.6211E-2 & -  & 7.4156E-3 & -      \\
		0.05                          & 7.1679E-3 & 1.1773 & 2.7562E-3 & 1.4279  \\
		0.025                         & 3.3534E-3 & 1.0959 & 1.1436E-3 & 1.2691 \\
		0.0125                         & 1.6194E-3 & 1.0501 & 5.1351E-4 & 1.1551  \\
		\hline	\hline
	\end{tabular}
	\caption{The $\ell ^ {\infty}$ error and convergence order for the numerical solutions of $c^1$ and $c^2$ up to $T=1$ with $\Delta t=h/10$ for the semi-implicit scheme~\reff{disc} and ~\reff{D*1}-\reff{D*2}.  {The tolerance $\varepsilon_{\text{tol}}=1$E$-6$.}}
	\label{table-h}	
\end{table}

\begin{table}[H]
	\centering
	\setlength{\abovecaptionskip}{10pt}%
	\setlength{\belowcaptionskip}{0pt}%
	\begin{tabular}{ccccc}
		\hline \hline
		$h$  & $c^1$  & Order  & $c^2$    & Order   \\
		\hline
		0.1                              & 1.6211E-2 & -  & 7.4156E-3 & -    \\
		0.05                             & 4.0353E-3 & 2.0062 & 1.8320E-3     & 2.0171 \\
		0.025                            & 1.0077E-3 & 2.0015 & 4.5664E-4 & 2.0043 \\
		0.0125                          & 2.5187E-4 & 2.0004 & 1.1408E-4 & 2.0011  \\
		\hline \hline
	\end{tabular}
   \caption{The $\ell ^ {\infty}$ error and convergence order for the numerical solutions of $c^1$ and $c^2$ up to $T=1$ with $\Delta t=h^2$ for the semi-implicit scheme~\reff{disc} and ~\reff{D*1}-\reff{D*2}. {The tolerance $\varepsilon_{\text{tol}}=1$E$-6$ for $h\geq 0.025$ and $\varepsilon_{\text{tol}}=1$E$-7$ for $h=0.0125$.}}
	\label{table-h2}	
\end{table}

As shown in Table~\ref{table-h}, the numerical error decreases robustly as the grid size refines with $\Delta t=h/10$ and the numerical convergence order is about one. From Table~\ref{table-h2}, one can see that the numerical error decreases much faster with $\Delta t=h^2$ and the numerical convergence order is about two. Such numerical tests clearly demonstrate that the proposed numerical scheme~\reff{disc} and \reff{D*1}-\reff{D*2}, as expected, is first-order accurate in temporal discretization and second-order accurate in spatial discretization. {Note that in order to see the convergence order for refined meshes, the last tolerance $\varepsilon_{\text{tol}}$ for the local curl-free relaxation is $10^{-7}$, which is more stringent than others.}

\subsection{Structure-preserving properties and applications}
In this section, we perform a series of numerical experiments to demonstrate that the proposed numerical scheme~\reff{disc} and \reff{D*1}-\reff{D*2}, combined with the local curl-free algorithm, can preserve solution positivity, mass conservation, and energy dissipation of the MANP equations at the discrete level. 
In numerical tests, we consider the charge dynamics of binary monovalent electrolytes governed by the same equations as in Section \ref{accuracy}, but with the fixed charge distribution given by
\[
\rho^{f}={\bf 1}_{\{ 0.24 \leq r^2\leq0.26, ~0<\theta\leq\pi\}} - {\bf 1}_{\{ 0.24 \leq r^2\leq0.26, ~\pi < \theta\leq2\pi\}},
\]
which is expressed in the polar coordinate. Such a setting describes the charge dynamics around a Janus sphere with permanent positive and negative charges on the upper and lower hemispheres, respectively. We consider a variable dielectric coefficient given by
\[
\varepsilon(x, y)=
	\frac{\varepsilon_w-\varepsilon_m}{2}\left[\tanh \left(50 \sqrt{x^{2}+y^{2}}-25\right)+1\right]+\varepsilon_m.
\]
The initial conditions are given by
$$
c^\ell(0,x,y)=0.1,~~\ell=1,2.
$$
To consider steric effects and Born solvation interactions, we take the correlation chemical potential as
$$\mu^{\ell,\text{cr}}= -\frac{v^{\ell}}{v^0} \log(v^0c^0)+ \chi \frac{(q^\ell)^2}{a^\ell}  \left(\frac{1}{\varepsilon} -1\right),$$ and parameters  $v^1=0.716^3$, $v^2=0.676^3$, $v^0=0.275^3$, and $\chi=198.9437$. {The tolerance in the local curl-free relaxation is set as $\varepsilon_{\text{tol}}=10^{-5}$.}

\begin{figure}[ht]
	\centering
	\includegraphics[scale=0.45]{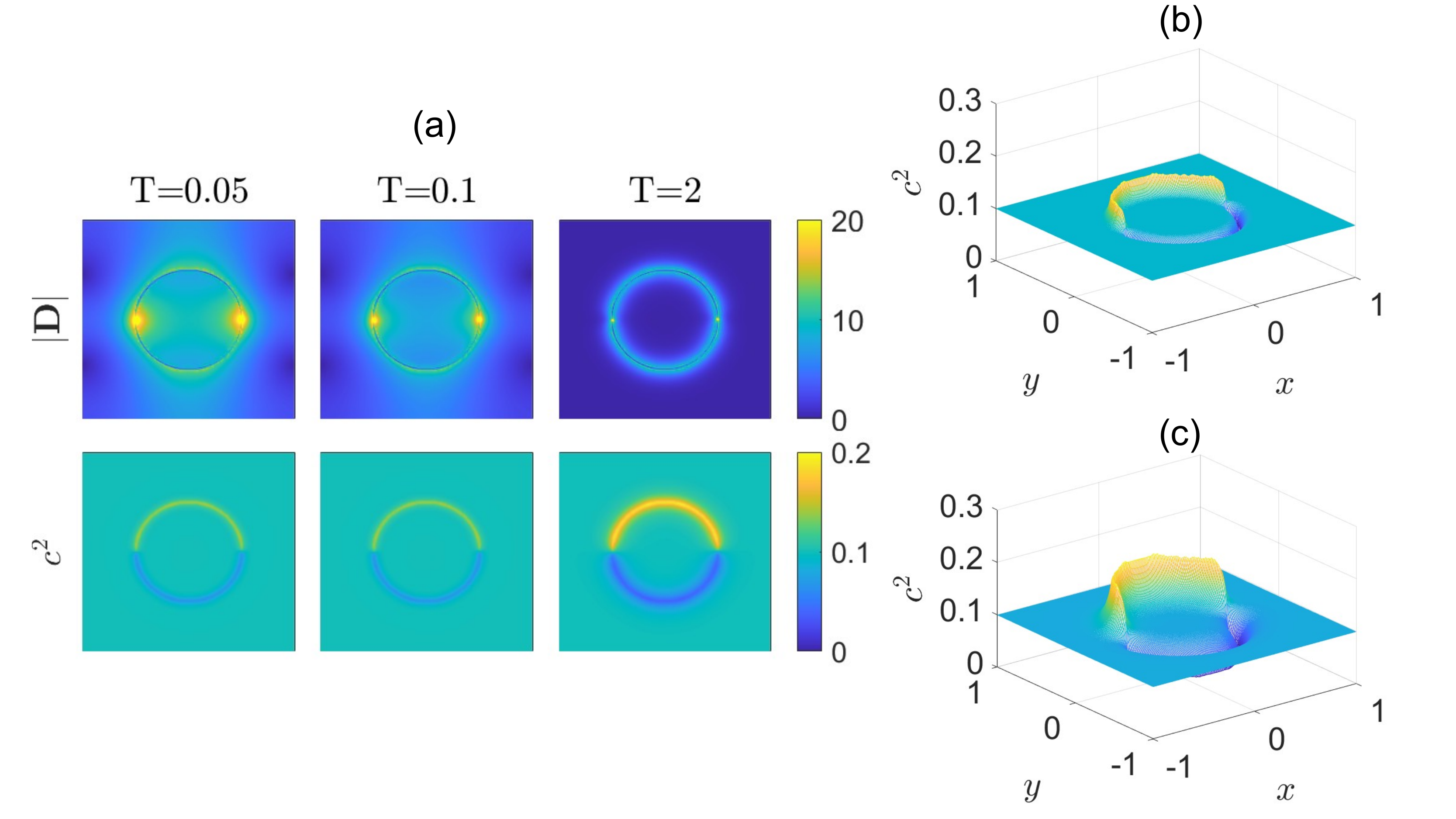}
	\caption{(a) The snapshots of $|\bm{D}|$ and $c^2$ at time $T=0.05$, $T=0.1$, and $T=2$ with $\kappa=0.02$ and $\varepsilon_m=\varepsilon_w=1$; (b) Plot of $c^2$ at time $T=0.05$; (c) Plot of $c^2$ at time $T=2$. }
	\label{f:vem=vew}
\end{figure}

In the first case, we consider a spatially uniform dielectric coefficient with $\varepsilon_m= \varepsilon_w =1$ and $\kappa=0.02$. The numerical simulations are performed with $h=0.01$ and $\Delta t= h^2$. Fig.~\ref{f:vem=vew} (a) displays the snapshots of $|\bm{D}|$ and concentrations of $c^2$ at time $T=0.05$, $T=0.1$, and $T=2$.  One can observe from Fig.~\ref{f:vem=vew} (b) and (c) that ions under electrostatic interactions gradually accumulate at the fixed charges with opposite signs in a very thin ring region, because the relatively small coefficient $\kappa$ gives rise to the boundary layers of ions near the fixed charges.  At the same time, the magnitude of the electric displacement first develops two peaks at positions where fixed charges undergo sharp change of signs. As counterions approach, the magnitude gets screened quickly as time evolves.

\begin{figure}[htbb]
	\centering
	\includegraphics[scale=0.45]{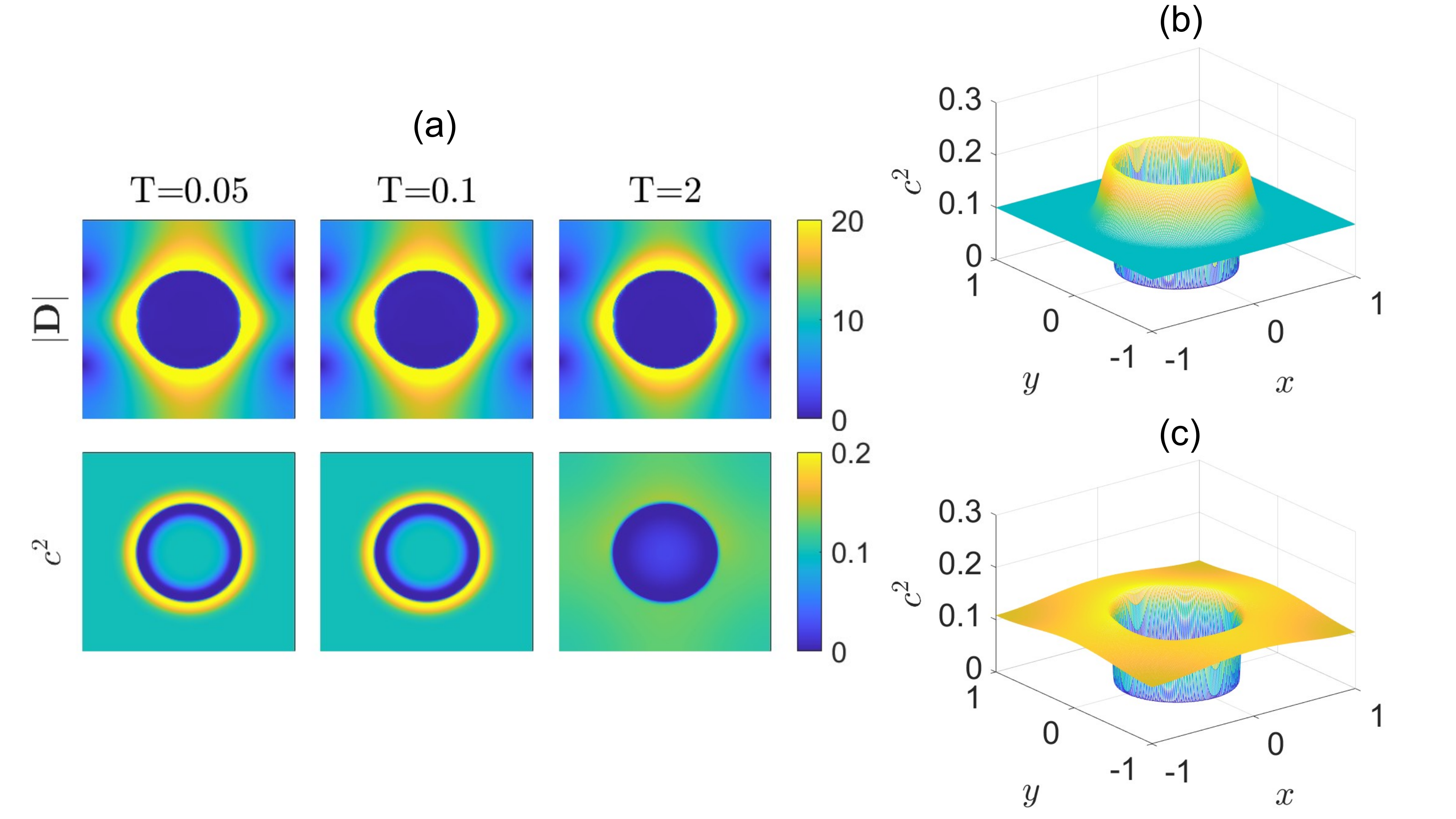}
	\caption{(a) The snapshots of $|\bm{D}|$ and $c^2$ at time $T=0.05$, $T=0.1$, and $T=2$ with $\kappa=0.02$, $\varepsilon_m=1$ and  $\varepsilon_w=78$; (b) Plot of $c^2$ at time $T=0.05$; (c) Plot of $c^2$ at time $T=2$. }
	\label{f:k001}
\end{figure}

\begin{figure}[htbb]
	\centering
	\includegraphics[scale=0.45]{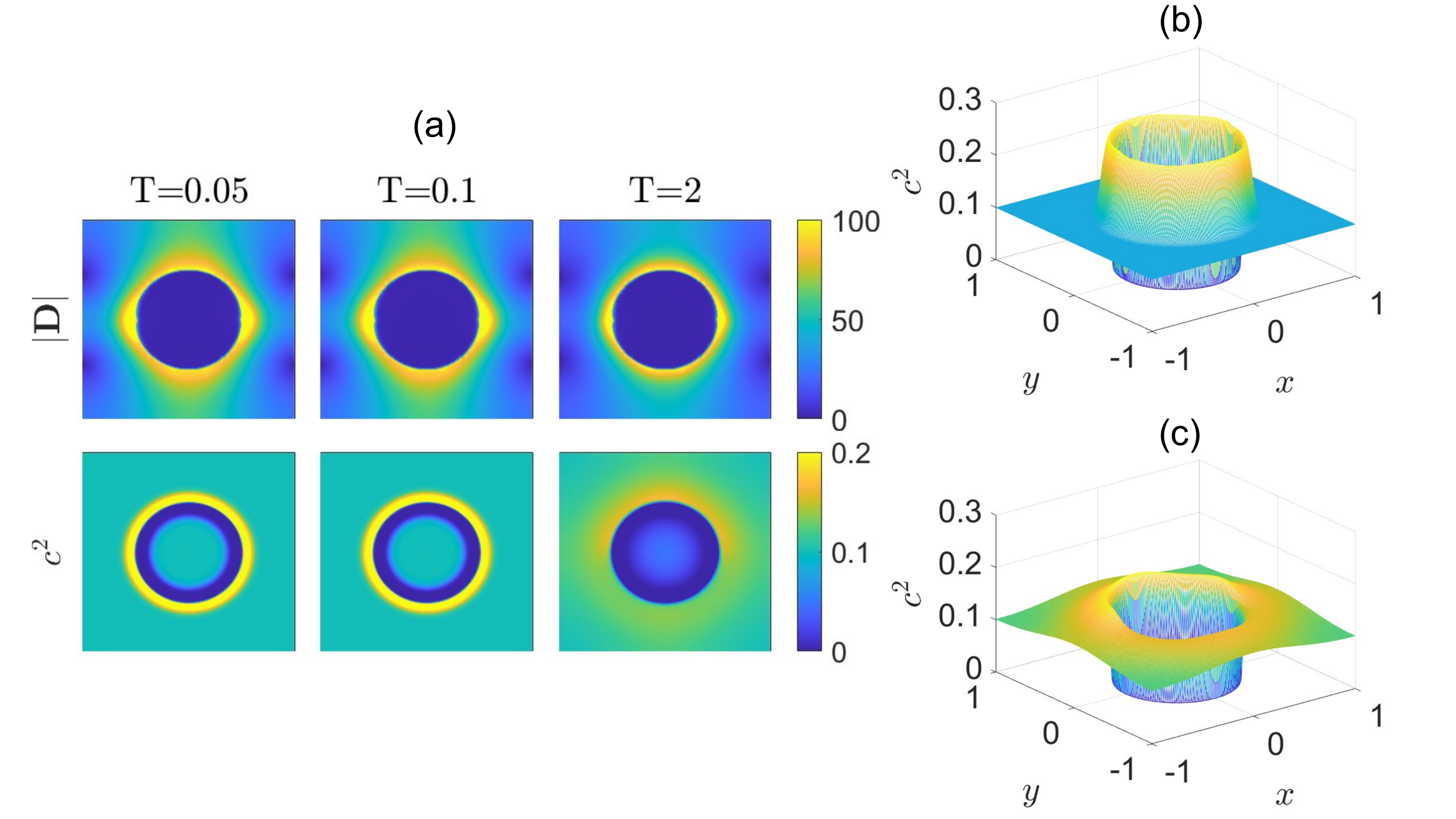}
	\caption{(a) The snapshots of $|\bm{D}|$ and $c^2$ at time $T=0.05$, $T=0.1$, and $T=2$ with $\kappa=0.01$, $\varepsilon_m=1$ and  $\varepsilon_w=78$; (b) Plot of $c^2$ at time $T=0.05$; (c) Plot of $c^2$ at time $T=2$. }
	\label{f:k005}
\end{figure}

In the second case, we perform numerical simulations using variable dielectric coefficients with $\varepsilon_m=1$ and $\varepsilon_w=78$, which mimic dielectric contrast between Janus sphere and water.  In contrast to the previous case, the dielectric mismatch causes large convection via the Born solvation interactions.  From Fig.~\ref{f:k001}, one can observe that Born solvation interactions dominate over other effects in the ion dynamics, resulting in a completely different pattern of ionic distributions. Due to dielectric mismatch, Born solvation interactions drive both cations and anions from low dielectric region to high dielectric region, quickly forming ring structure around fixed charges in the first phase of ion dynamics. After that, the ring structure eventually diffuses over the whole region of high dielectrics. Meanwhile, the electric displacement slightly gets screened by the mobile ions.



In the third case, to evaluate the effect of $\kappa$, we investigate the ionic dynamics with a smaller $\kappa=0.01$. As shown in Fig.~\ref{f:k005}, both cations and anions again are quickly expelled from a low-dielectric region to a high-dielectric region, indicating that the Born solvation interactions still dominate over other effects in the first phase. As the system approaches the steady state, the counterions are still attracted around the oppositely charged hemisphere of the Janus sphere, resulting from the competitions among various effects.  Comparing with the previous case, one can observe that the smaller $\kappa$ leads to a sharper boundary-layer structure near charged surface.

Large convection has been a challenging numerical difficulty to deal with in the study of ion dynamics. There are several factors could contribute to the large convection in the system. For instance, when the dimensionless coefficient $\kappa$, the ratio of the Debye length to spatial characteristic length, is small, the emergence of boundary layers could give rise to large electric displacement in the MANP equations, as shown in the third case.  In addition, the steric effects and Born solvation interactions could cause large convection in regions where there is a dielectric mismatch or concentrated ionic solution. These could happen in the vicinity of the fixed charges in the system under consideration.  To demonstrate the effectiveness of the numerical scheme, we characterize the convection quantitatively by the numerical cell P\'{e}clet number, which is a dimensionless number defined by
$$
{\mbox Pe}^{\ell, n}_{i,j}=\max \left\{|dg^{\ell,n}_{i+\frac{1}{2},j}|, ~|dg^{\ell,n}_{i,j+\frac{1}{2}}|\right\}.
$$

\begin{figure}[htbp]
	\centering
	\includegraphics[scale=0.8]{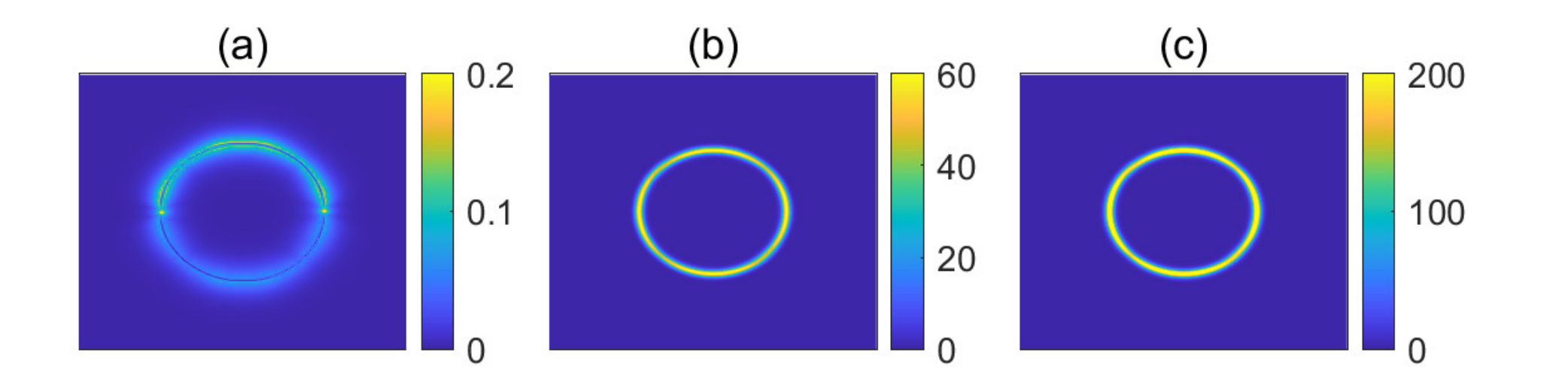}
	\caption{The distribution of the numerical cell P\'{e}clet number. (a) $\ve_m=\ve_w=1$ and $\kappa=0.02$; (b) $\ve_m=1$, $\ve_w=78$ and $\kappa=0.02$;
	(c) $\ve_m=1$, $\ve_w=78$ and $\kappa=0.01$.}
	\label{f:Pe}
\end{figure}

Fig.~\ref{f:Pe} plots the distribution of the numerical cell P\'{e}clet number for previous three cases. Clearly, one can see that the high cell P\'{e}clet number mainly concentrates around the fixed charges. Asymmetry between upper and lower halves is ascribed to the different ionic volumes and signs of fixed charges.  Dielectric mismatch, which accounts for the Born solvation interactions, contributes significantly to the convection. Small $\kappa$, associated with the boundary-layer structure, is another source of the large convection. Stable numerical results presented in previous three cases demonstrate that the developed numerical scheme can robustly handle large convection of high numerical cell P\'{e}clet numbers.


\begin{figure}[h!]
	\centering
	\includegraphics[scale=0.65]{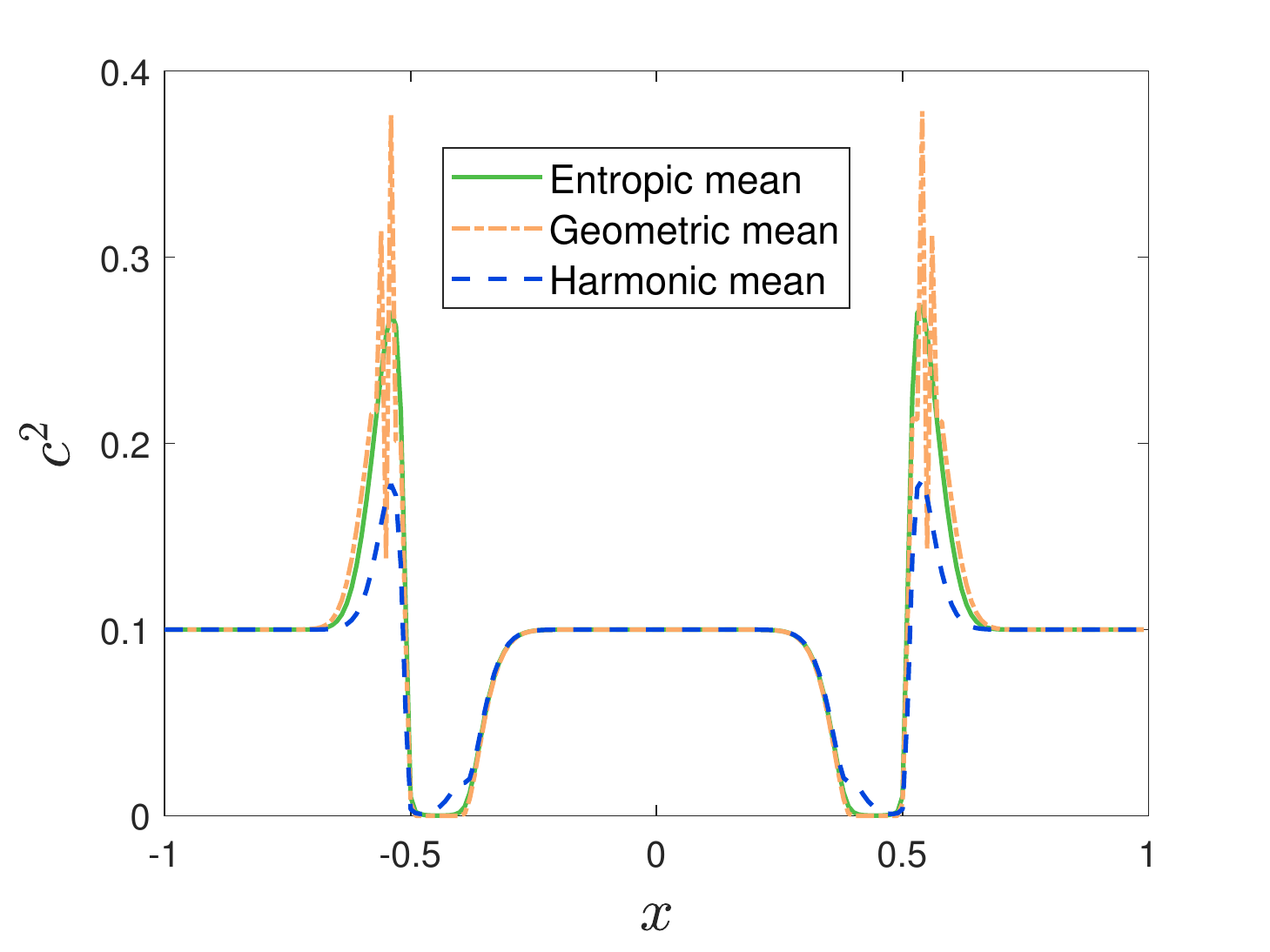}
	\caption{{Cross-section snapshots at $y = 0$ of anion concentrations with different means at time $T=0.05$. Note that the results with the standard central difference and arithmetic mean are not displayed due to the presence of non-positive numerical values for $c^{\ell}$ or $c^0$. }}
	\label{f:means}
\end{figure}

{
In addition, we perform numerical tests to compare the performance of the developed scheme based on the entropic mean, the standard central differencing scheme, and schemes obtained by using other mean approximations mentioned in Remark~\ref{Means} with $h=0.01$ and $\Delta t=0.001$. We apply the schemes to numerically simulate the tough case with $\varepsilon_m=1$, $\varepsilon_w=78$ and $\kappa=0.01$, which has large convection as demonstrated in Fig.~\ref{f:Pe}.  We find that the simulations with standard central differencing scheme and the arithmetic-mean scheme both fail for the presence of non-positive numerical concentrations. From Fig.~\ref{f:means}, one can see that the oscillations emerge around the interface for the geometric-mean scheme. The one predicted by the harmonic mean fails to capture the concentration variations at the interface. Such results demonstrate that the developed numerical scheme with entropic mean approximation can robustly handle large convection in charge transport. 
}

\begin{figure}[htbb]
	\centering
	\includegraphics[scale=0.35]{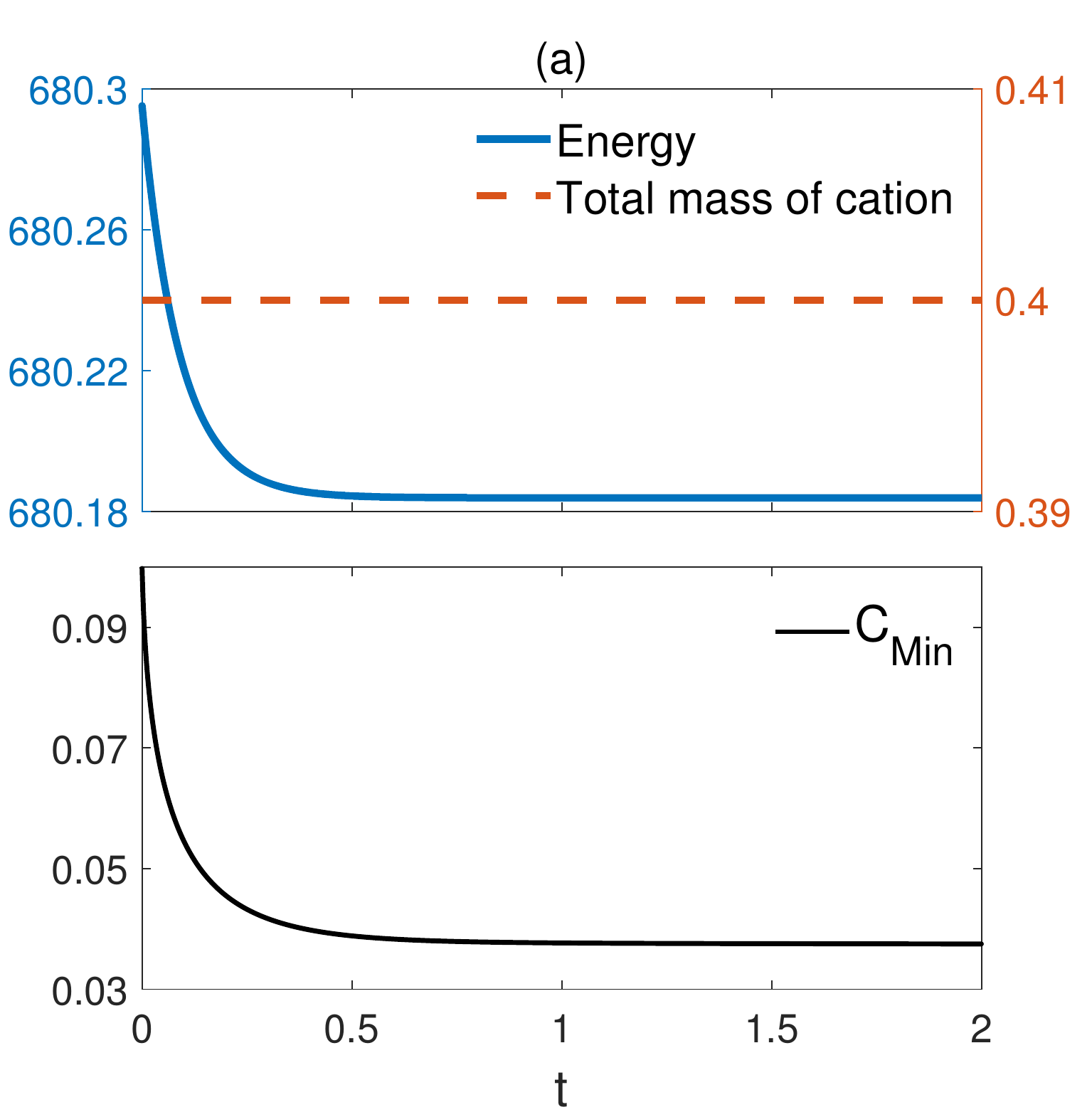} \hspace{-3mm}
	\includegraphics[scale=0.35]{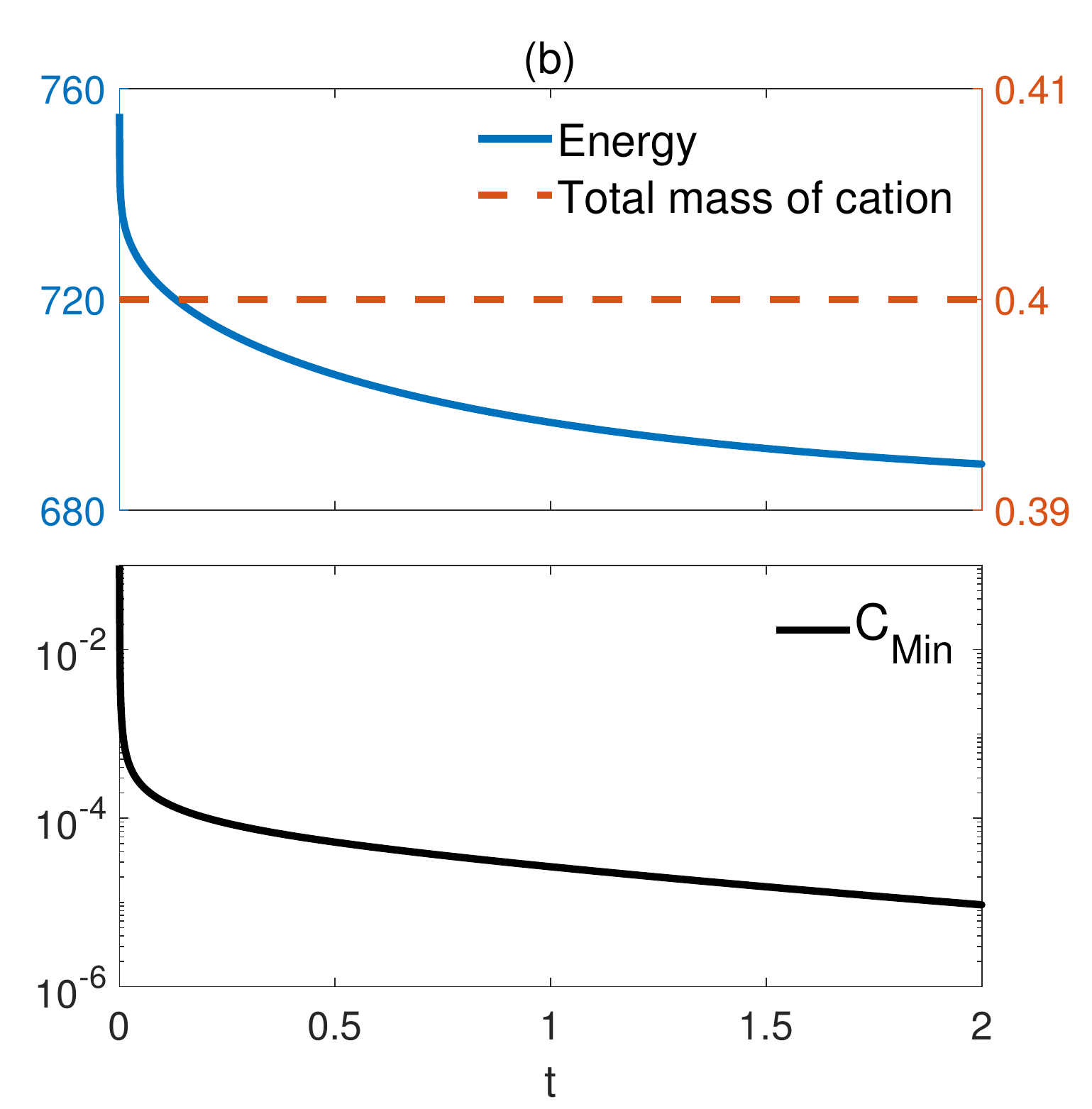} \hspace{-3mm}
	\includegraphics[scale=0.35]{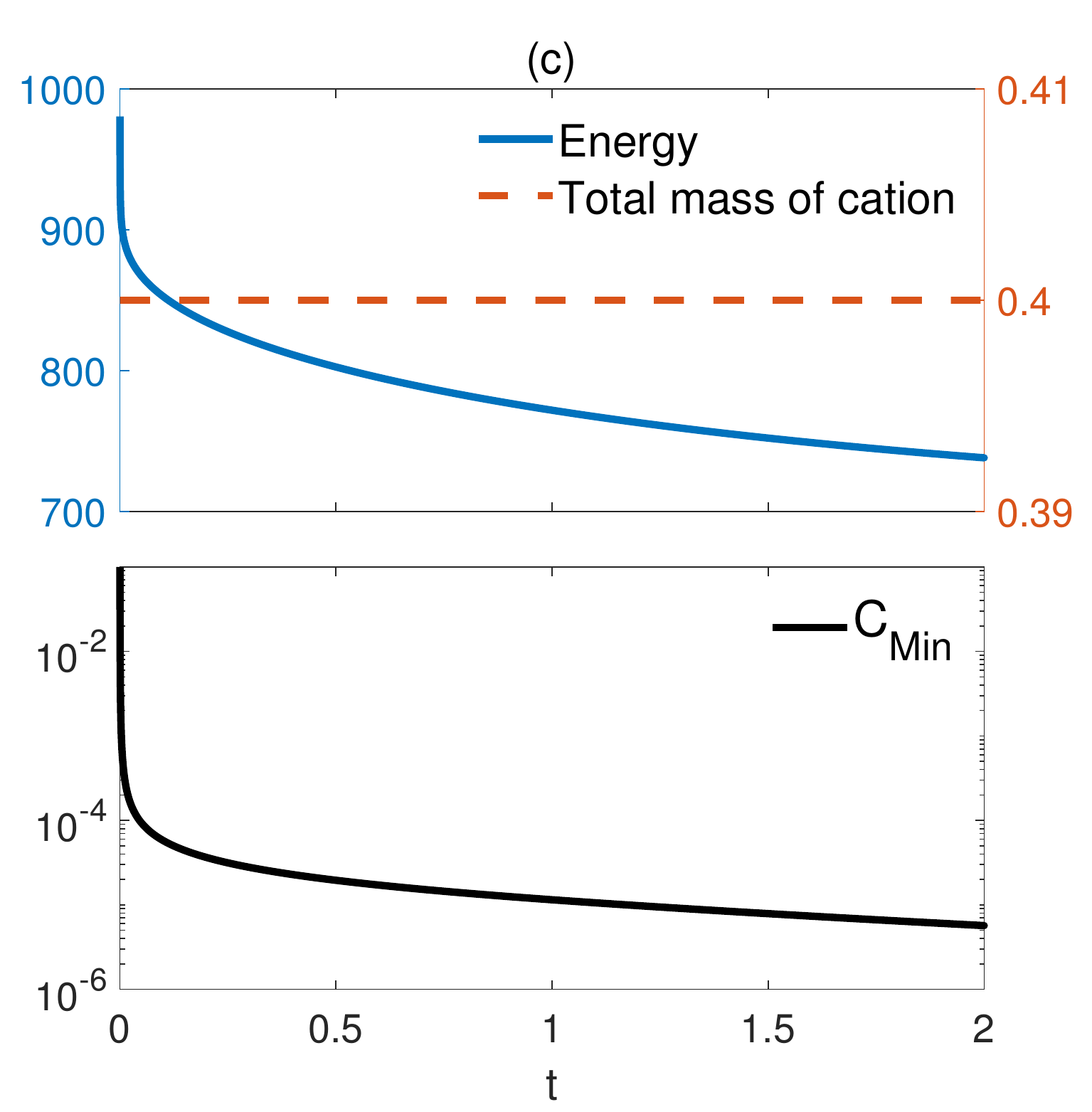}
	\caption{The evolution of total mass, discrete energy $\mathcal{F}_h$, and the minimum concentration on the computational grid. (a) $\ve_m=\ve_w=1$ and $\kappa=0.02$; (b) $\ve_m=1$, $\ve_w=78$ and $\kappa=0.02$;
	(c) $\ve_m=1$, $\ve_w=78$ and $\kappa=0.01$.}
	\label{f:prop}
\end{figure}
{Furthermore}, we perform numerical tests on the performance of the developed numerical scheme in preserving the physical properties at the discrete level. Fig.~\ref{f:prop} depicts the evolution of total mass, discrete energy $\mathcal{F}_h$, and the minimum concentration values on the computational grid, $C_{\rm Min}$, for previous three cases. Clearly, the total mass of concentrations is perfectly conserved during the time evolution.  Of much interest is the evolution of the minimum concentration, which illustrates that the numerical solution of both concentrations {remains} positive all the time. Being consistent with the numerical analysis, the discrete energy $\mathcal{F}_h$ decreases monotonically. From the decay rate, one can observe that the system with dielectric mismatch relaxes in a smaller timescale in the first phase of ion dynamics, but it eventually takes much longer time to reach the steady state. Such numerical results verify that the developed numerical scheme can effectively preserve the physical properties at the discrete level.

\begin{figure}[h!]
	\centering
	\includegraphics[scale=0.6]{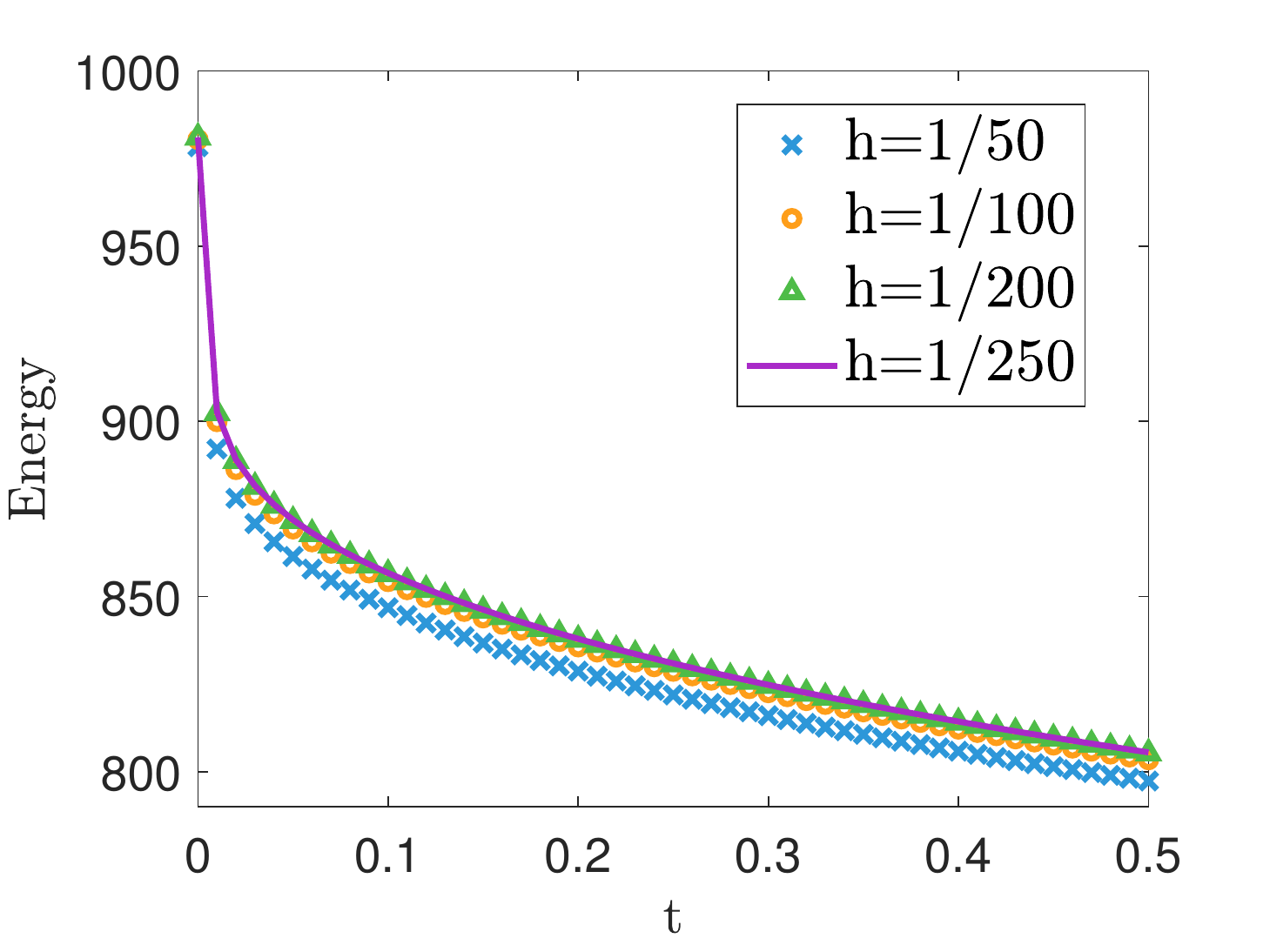}
	\caption{{The evolution of discrete energy $\mathcal{F}_h$ with a fixed time step $\Delta t=0.01$ and different grid spacing with $\ve_m=1$, $\ve_w=78$ and $\kappa=0.01$.}}
	\label{f:Energy}
\end{figure}

{As revealed in Theorem 3.3, there is a sufficient time-step constraint $\Delta t <\Delta t^*$ for the discrete energy dissipation. To further explore the time-step constraint,  we perform a series of numerical simulations with a fixed time step $\Delta t=0.01$ and refine the grid spacing from $h=1/50$ to $h=1/250$. As shown in Fig.~\ref{f:Energy}, the discrete energy keeps decreasing monotonically with various mesh resolution. This demonstrates that the constraint on time step is insensitive to the spatial resolution.   
}

\begin{figure}[h!]
	\centering
	\includegraphics[scale=0.6]{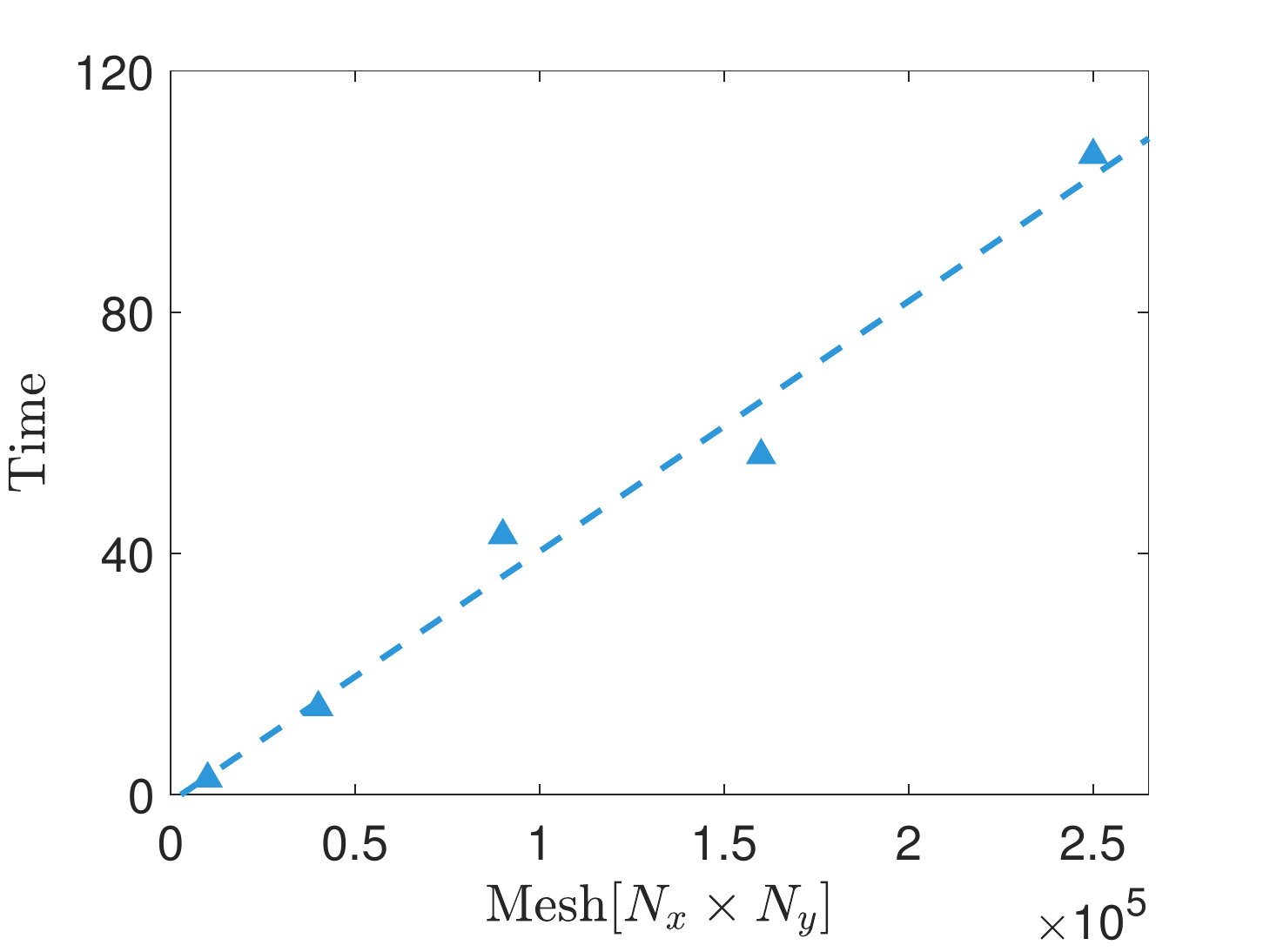}
	\caption{{The computational time of the local curl-free relaxation (shown in triangles) with $\ve_m=1$, $\ve_w=78$ and $\kappa=0.01$. The dash line is shown for reference.}}
	\label{f:LocalTime}
\end{figure}
{To further demonstrate the computational complexity, we perform numerical simulations with a variable dielectric coefficient and count the computation time spent in the local curl-free relaxation. Fig.~\ref{f:LocalTime} presents the computational time against the mesh size. Clearly, one can see that the local curl-free relaxation algorithm has linear computational complexity.}

\section{Conclusions}\label{s:Con}
In this work, we have developed a structure-preserving numerical scheme for an MANP model, which describes ion dynamics via ionic concentrations and the electric displacement. In such an MANP model, effects beyond mean-field approximations, such as ionic steric effects and Born solvation interactions, have been taken into account. The solution to the MANP model has been shown to have several intrinsic properties of importance. Based on the Slotboom transformation, a structure-preserving semi-implicit numerical scheme for the generalized Nernst-Planck equations has been proposed. To deal with the curl-free constraint, the dielectric displacement obtained from the Maxwell-Amp\`{e}re equation has been further updated with a convergent local relaxation algorithm of linear computational complexity. {Given its high efficiency, the local relaxation algorithm provides a promising alternative for numerically solving the Poisson's equation,
especially for the case with a variable coefficient.} The whole proposed numerical algorithm has been proved to preserve mass conservation, solution positivity, and energy dissipation of the MANP model at the discrete level. Such structure-preserving properties, as well as numerical accuracy, have been thoroughly tested in numerical simulations. Furthermore, numerical tests on charge dynamics of electrolyte solutions have demonstrated that the proposed numerical scheme can effectively handle large convection, arising from boundary-layer electric displacement and Born solvation interactions, up to high numerical cell P\'{e}clet numbers.

\setcounter{equation}{0}
\setcounter{subsection}{0}
\renewcommand{\theequation}{A.\arabic{equation}}
\renewcommand{\thesubsection}{A.\arabic{subsection}}
\section*{Appendix: Proof of Theorem ~\ref{thm4}}
\appendix

The following proof follows closely the idea presented in the work~\cite{LM2020}, with modifications on the treatment of the electric displacement and entropic-mean approximation in the fluxes.
\begin{proof}

It follows from~\reff{energydis} and the assumption on the time-independence of $\mu^{\ell,\text{cr}}$ that
\begin{equation}
			\begin{aligned}
				\mathcal{F}^{n+1}_h -\mathcal{F}^n_h=
				&\Delta \Omega \sum_{\ell=1}^{M} \sum_{i,j} \left[(c^{\ell,n+1}_{i,j}-c^{\ell,n}_{i,j})\left(\mu^{\ell,*}_{i,j}-q^\ell \phi_{i,j}^n\right)+   c^{\ell,n}_{i,j} \log\frac{c^{\ell,n+1}_{i,j}}{c^{\ell,n}_{i,j}} \right]\\
				& +\Delta \Omega \kappa^2 \sum_{i,j} \left(\frac{\left|D^{n+1}_{i+\frac{1}{2},j}\right|^2-\left|D^{n}_{i+\frac{1}{2},j}\right|^2}{\varepsilon_{i+\frac{1}{2},j}}+ \frac{\left|D^{n+1}_{i,j+\frac{1}{2}}\right|^2-\left|D^{n}_{i,j+\frac{1}{2}}\right|^2}{\varepsilon_{i,j+\frac{1}{2}}} \right),
			\end{aligned}
\end{equation}
		where the chemical potential	
		$$
		\mu^{\ell,*}_{i,j}=\left(g^{\ell,n}_{i,j}+  \log c^{\ell,n+1}_{i,j} \right).
		$$
		We have by mass conservation that
		$$
		\sum_{i,j} \Delta \Omega c^{\ell,n}_{i,j} \log\frac{c^{\ell,n+1}_{i,j}}{c^{\ell,n}_{i,j}} \leq \sum_{i,j} \Delta \Omega (c^{\ell,n+1}_{i,j}-c^{\ell,n}_{i,j})=0.
		$$
		Then, we have
		\begin{equation}
			\mathcal{F}^{n+1}_h-\mathcal{F}^n_h\leq -\Delta t I_1 +\Delta t^2 I_2,
		\end{equation}
		where
		$$
		I_1=-\sum_{\ell=1}^{M} \sum_{i,j} \Delta \Omega \left[ \frac{c^{\ell,n+1}_{i,j}-c^{\ell,n}_{i,j}}{\Delta t} \right] \mu_{i,j} ^{\ell,*},
		$$
		$$
		I_2=\sum_{i,j}\frac{\Delta \Omega }{\Delta t ^2} \left[\kappa^2\frac{\left|D^{n+1}_{i+\frac{1}{2},j}\right|^2-\left|D^{n}_{i+\frac{1}{2},j}\right|^2}{\varepsilon_{i+\frac{1}{2},j}}+\kappa^2 \frac{\left|D^{n+1}_{i,j+\frac{1}{2}}\right|^2-\left|D^{n}_{i,j+\frac{1}{2}}\right|^2}{\varepsilon_{i,j+\frac{1}{2}}} - \phi_{i,j}^n \sum_{\ell=1}^M  q^\ell (c^{\ell,n+1}_{i,j}-c^{\ell,n}_{i,j})\right].
		$$
		By the numerical scheme~\eqref{disc} and summation by parts, we have
		\begin{equation}
			\begin{aligned}
				I_1&=\sum_{\ell=1}^{M} \sum_{i,j} \Delta \Omega \left( \frac{J^{\ell,n}_{i+\frac{1}{2},j} - J^{\ell,n}_{i-\frac{1}{2},j}}{\Delta x} + \frac{J^{\ell,n}_{i,j+\frac{1}{2}} - J^{\ell,n}_{i,j-\frac{1}{2}}}{\Delta y} \right) \mu_{i,j} ^{\ell,*}\\
				&=- \sum_{\ell=1}^{M} \sum_{i,j} \Delta \Omega \left( J^{\ell,n}_{i+\frac{1}{2},j} \frac{\mu^{\ell,*}_{i+1,j}-\mu^{\ell,*}_{i,j}}{\Delta x} + J^{\ell,n}_{i,j+\frac{1}{2}} \frac{\mu^{\ell,*}_{i,j+1}-\mu^{\ell,*}_{i,j}}{\Delta y} \right).
			\end{aligned}\label{I}
		\end{equation}
		Since
		$$
		J_{i+\frac{1}{2},j}^{\ell,n}=-\kappa e^{-g^{\ell,n}_{i+\frac{1}{2},j}}\frac{e^{ \mu^{\ell,*}_{i+1,j}}-e^{ \mu^{\ell,*}_{i,j}}}{\Delta x}~~\text{and}~~
		J_{i,j+\frac{1}{2}}^{\ell,n}=-\kappa e^{-g^{\ell,n}_{i,j+\frac{1}{2}}}\frac{e^{\mu^{\ell,*}_{i,j+1}}-e^{\mu^{\ell,*}_{i,j} }}{\Delta y},
		$$
		we have $I_1 \geq 0$. It follows from the discrete Gauss's law~\eqref{gausslaw} that
		\[
		\begin{aligned}
		I_2=&\sum_{i,j}\frac{\Delta \Omega}{\Delta t ^2}  \left[\kappa^2\frac{\left|D^{n+1}_{i+\frac{1}{2},j}\right|^2-\left|D^{n}_{i+\frac{1}{2},j}\right|^2}{\varepsilon_{i+\frac{1}{2},j}}+\kappa^2 \frac{\left|D^{n+1}_{i,j+\frac{1}{2}}\right|^2-\left|D^{n}_{i,j+\frac{1}{2}}\right|^2}{\varepsilon_{i,j+\frac{1}{2}}} \right. \\
		&\left.- 2\kappa^2\phi_{i,j}^n \Bigg( \frac{D^{n+1}_{i+\frac{1}{2},j}-D^{n+1}_{i-\frac{1}{2},j} - D^{n}_{i+\frac{1}{2},j}+D^{n}_{i-\frac{1}{2},j} }{\Delta x}
		+\frac{D^{n+1}_{i,j+\frac{1}{2}}-D^{n+1}_{i,j-\frac{1}{2}} - D^{n}_{i,j+\frac{1}{2}}+D^{n}_{i,j-\frac{1}{2}} }{\Delta y} \Bigg) \right].
		\end{aligned}
		\]
		Notice that both $\bm{D}^*$ and $\bm{D}^{n+1}$ satisfy the {discrete} Gauss's law, according to the description of the local curl-free algorithm in Section~\ref{localsec}. By summation by parts and \eqref{dphi}, we have
		\begin{equation}
				I_2
				=\sum_{i,j} \Delta \Omega \kappa^2 \left[  \frac{\left|D^{n+1}_{i+\frac{1}{2},j}-D^{n}_{i+\frac{1}{2},j}\right|^2}{ \varepsilon_{i+\frac{1}{2},j}\big(\Delta t\big) ^2} + \frac{\left|D^{n+1}_{i,j+\frac{1}{2}}-D^{n}_{i,j+\frac{1}{2}}\right|^2}{ \varepsilon_{i,j+\frac{1}{2}}\big(\Delta t\big)^2} \right]  \geq 0.
		\end{equation}
		The energy dissipation inequality~\eqref{energyinequ} is satisfied if
		\begin{equation}
			\Delta t \leq \Delta t^{*} \leq \frac{I_1}{2 I_2},
		\end{equation}	
		where the critical time step size $\Delta t^*$ can be obtained by estimating the lower bound of $ I_1/2 I_2$.
		From the discrete Gauss's law~\eqref{gausslaw}, we get
		\begin{equation}
			\begin{aligned}
			&2\kappa^2\frac{D^{n+1}_{i+\frac{1}{2},j}-D^{n+1}_{i-\frac{1}{2},j}-D^{n}_{i+\frac{1}{2},j}+D^{n}_{i-\frac{1}{2},j}}{\Delta x} + 2\kappa^2\frac{D^{n+1}_{i,j+\frac{1}{2}}-D^{n+1}_{i,j-\frac{1}{2}}-D^{n}_{i,j+\frac{1}{2}}+D^{n}_{i,j-\frac{1}{2}}}{\Delta y}\\ &=\sum_{\ell=1}^M q^\ell  \left(c^{\ell,n+1}_{i,j}-c^{\ell,n}_{i,j} \right).
			\end{aligned}
		\end{equation}	
		Multiplication of $\left( \phi^{n+1}_{i,j}-\phi^n_{i,j} \right) \Delta \Omega$ and summation over $i,j$ lead to
		\begin{equation}\label{Dphi}
			\begin{aligned}
				&2\kappa^2 \sum_{i,j} \Delta \Omega  \left[ \frac{D^{n+1}_{i+\frac{1}{2},j}-D^{n+1}_{i-\frac{1}{2},j}-D^{n}_{i+\frac{1}{2},j}+D^{n}_{i-\frac{1}{2},j}}{\Delta x} + \right.\\ &\qquad \qquad \qquad \left.\frac{D^{n+1}_{i,j+\frac{1}{2}}-D^{n+1}_{i,j-\frac{1}{2}}-D^{n}_{i,j+\frac{1}{2}}+D^{n}_{i,j-\frac{1}{2}}}{\Delta y}\right]  \left(\phi^{n+1}_{i,j}-\phi^{n}_{i,j}\right)  \\
				&=\sum_{\ell=1}^M \sum_{i,j} \Delta \Omega q^\ell \left(c^{\ell,n+1}_{i,j}-c^{\ell,n}_{i,j} \right) \left(\phi^{n+1}_{i,j}-\phi^{n}_{i,j}\right).
			\end{aligned}
		\end{equation}
	      By summation by parts, the left hand side of~\reff{Dphi} becomes
		\begin{equation}
			\begin{aligned}
				LHS 
				&=2\kappa^2\sum_{i,j} \Delta \Omega  \left( \frac{\left|D^{n+1}_{i+\frac{1}{2},j}-D^{n}_{i+\frac{1}{2},j}\right|^2}{\varepsilon_{i+\frac{1}{2},j}} +\frac{\left|D^{n+1}_{i,j+\frac{1}{2}}-D^{n}_{i,j+\frac{1}{2}}\right|^2}{\varepsilon_{i,j+\frac{1}{2}}} \right) \\
				& \geq \frac{2\kappa^2}{\varepsilon_{\max}} \sum_{i,j} \Delta \Omega \left( \left|D^{n+1}_{i+\frac{1}{2},j}-D^{n}_{i+\frac{1}{2},j}\right|^2+ \left|D^{n+1}_{i,j+\frac{1}{2}}-D^{n}_{i,j+\frac{1}{2}}\right|^2 \right).
			\end{aligned} \label{lhs}
		\end{equation}
		By the numerical scheme~\eqref{disc} and summation by parts, the right hand side of~\reff{Dphi} becomes
		\begin{equation}
			\begin{aligned}
				RHS&=- \Delta t \sum_{\ell=1}^M \sum_{i,j} \Delta \Omega q^\ell  \left( \frac{J^{\ell,n}_{i+\frac{1}{2},j}-J^{\ell,n}_{i-\frac{1}{2},j}}{\Delta x} +\frac{J^{\ell,n}_{i,j+\frac{1}{2}}-J^{\ell,n}_{i,j-\frac{1}{2}}}{\Delta y} \right) \left(\phi^{n+1}_{i,j}-\phi^{n}_{i,j}\right)\\
				&=- \Delta t \sum_{\ell=1}^M \sum_{i,j} \Delta \Omega q^\ell \left( J^{\ell,n}_{i+\frac{1}{2},j}  \frac{D^{n+1}_{i+\frac{1}{2},j}-D^{n}_{i+\frac{1}{2},j}}{\varepsilon_{i+\frac{1}{2},j}}  + J^{\ell,n}_{i,j+\frac{1}{2}}  \frac{D^{n+1}_{i,j+\frac{1}{2}}-D^{n}_{i,j+\frac{1}{2}}}{\varepsilon_{i,j+\frac{1}{2}}} \right).
			\end{aligned}
		\end{equation}
		By the Cauchy-Schwarz inequality, we have
		\begin{equation}
			\begin{aligned}
				RHS &\leq \Delta t \sum_{\ell=1}^M \left|q^\ell \right|
				 \left\{ \sum_{i,j} \Delta \Omega \left[ \left(J^{\ell,n}_{i+\frac{1}{2},j}\right)^2+ \left(J^{\ell,n}_{i,j+\frac{1}{2}}\right)^2 \right] \right\}^{\frac{1}{2}} \\
				&\qquad \cdot \left\{ \sum_{i,j} \Delta \Omega \left[ \left( \frac{D^{n+1}_{i+\frac{1}{2},j}-D^{n}_{i+\frac{1}{2},j}}{\varepsilon_{i+\frac{1}{2},j}} \right)^2 + \left( \frac{D^{n+1}_{i,j+\frac{1}{2}}-D^{n}_{i,j+\frac{1}{2}}}{\varepsilon_{i,j+\frac{1}{2}}} \right)^2 \right]\right\} ^{\frac{1}{2}}.
			\end{aligned}
			\label{rhs}
		\end{equation}	
		Combination of~\eqref{lhs} and \eqref{rhs} leads to
		\begin{equation}
			\begin{aligned}
				&\sum_{i,j} \Delta \Omega \left( \left| D^{n+1}_{i+\frac{1}{2},j}-D^{n}_{i+\frac{1}{2},j} \right|^2 + \left| D^{n+1}_{i,j+\frac{1}{2}}-D^{n}_{i,j+\frac{1}{2}} \right|^2 \right)\\
				&\leq \frac{\Delta t^2 \varepsilon^2_{\max}}{4\kappa^4\varepsilon^2_{\min}}
				\left\{\sum_{\ell=1}^{M} \left| q^\ell \right| \left[ \sum_{i,j} \Delta \Omega \left[  \left(J^{\ell,n}_{i+\frac{1}{2},j}\right)^2+  \left(J^{\ell,n}_{i,j+\frac{1}{2}}\right)^2 \right] \right]^{\frac{1}{2}}  \right\}^2\\
				&\leq \frac{\Delta t^2 \varepsilon^2_{\max}}{4\kappa^4\varepsilon^2_{\min}} \sum_{\ell=1}^{M} \left| q^\ell \right|^2 \sum_{\ell=1}^{M} \sum_{i,j} \Delta \Omega \left[ \left(J^{\ell,n}_{i+\frac{1}{2},j}\right)^2+ \left(J^{\ell,n}_{i,j+\frac{1}{2}}\right)^2\right].
			\end{aligned}
		\end{equation}
		Thus, we have
		\begin{equation}
			I_2 \leq C \sum_{\ell=1}^{M} \sum_{i,j} \Delta \Omega \left[ \left(J^{\ell,n}_{i+\frac{1}{2},j}\right)^2+ \left(J^{\ell,n}_{i,j+\frac{1}{2}}\right)^2\right], \label{II}
		\end{equation}
		where $C= \varepsilon^2_{\max} \sum _{\ell=1}^M \left| q^\ell \right|^2 /( 4\kappa^2 \varepsilon^3_{\min})$. Therefore, we have
		\[
			\begin{aligned}
				\frac{I_1}{2I_2}
				& \geq \frac{- \sum_{\ell=1}^{M} \sum_{i,j} \Delta \Omega \left[ J^{\ell,n}_{i+\frac{1}{2},j} \frac{\mu^{\ell,*}_{i+1,j}-\mu^{\ell,*}_{i,j}}{\Delta x} + J^{\ell,n}_{i,j+\frac{1}{2}} \frac{\mu^{\ell,*}_{i,j+1}-\mu^{\ell,*}_{i,j}}{\Delta y} \right]}{2C \sum_{\ell=1}^{M} \sum_{i,j} \Delta \Omega \left[ \left(J^{\ell,n}_{i+\frac{1}{2},j}\right)^2+ \left(J^{\ell,n}_{i,j+\frac{1}{2}}\right)^2\right] }\\
				& \geq \frac{1}{2C} \min _{i,j,\ell} \left\{\frac{\mu^{\ell,*}_{i+1,j}-\mu^{\ell,*}_{i,j}}{-\Delta x J^{\ell,n}_{i+\frac{1}{2},j}},\frac{\mu^{\ell,*}_{i,j+1}-\mu^{\ell,*}_{i,j}}{-\Delta y J^{\ell,n}_{i,j+\frac{1}{2}}} \right\}.
			\end{aligned}
		\]
		Application of the mean-value theorem leads to
		\[
				\frac{I_1}{2I_2}
				\geq
				\frac{1}{2C\kappa}  \min _{i,j,\ell} \left\{ \frac{1}{ e^{-g^{\ell,n}_{i+\frac{1}{2},j}}  e^{\theta   \mu^{\ell,*}_{i+1,j}+(1-\theta)  \mu^{\ell,*}_{i,j} }}, \frac{1}{ e^{-g^{\ell,n}_{i,j+\frac{1}{2}}}  e^{\alpha   \mu^{\ell,*}_{i,j+1}+(1-\alpha)  \mu^{\ell,*}_{i,j} }  } \right\},
		\]
		where $\theta$ and $\alpha \in (0,1)$. By the entropic mean of $e^{-g^{\ell,n}_{i+1/2,j}}$,
		we have
		$$
		\begin{aligned}
			\frac{1}{ e^{-g^{\ell,n}_{i+\frac{1}{2},j}}  e^{\theta   \mu^{\ell,*}_{i+1,j}+(1-\theta)  \mu^{\ell,*}_{i,j} } }
			&= \frac{e^{-\theta g^{\ell,n}_{i+1,j}+(\theta -1) g^{\ell,n}_{i,j}}}{ \left(c^{\ell,n+1}_{i+1,j}\right)^{\theta} \left(c^{\ell,n+1}_{i,j}\right)^{1-\theta} } \frac{ e^{g^{\ell,n}_{i+1,j}}- e^{g^{\ell,n}_{i,j}}}{g^{\ell,n}_{i+1,j}-g^{\ell,n}_{i,j}}\\
			&=  \frac{1}{ \left(c^{\ell,n+1}_{i+1,j}\right)^{\theta} \left(c^{\ell,n+1}_{i,j}\right)^{1-\theta} } \frac{e^{(1-\theta)\left( g^{\ell,n}_{i+1,j}-g^{\ell,n}_{i,j} \right)}-e^{-\theta\left( g^{\ell,n}_{i+1,j}-g^{\ell,n}_{i,j} \right)}} {g^{\ell,n}_{i+1,j}-g^{\ell,n}_{i,j}} \\
			&\geq \frac{e^{-\left|g_{i+1,j}^{\ell,n}-g_{i,j}^{\ell,n}\right|}}{c_{\max}}.
		\end{aligned}
		$$
		Similarly,
		$$
		\frac{1}{ e^{-g^{\ell,n}_{i,j+\frac{1}{2}}}  e^{\alpha   \mu^{\ell,*}_{i,j+1}+(1-\alpha)  \mu^{\ell,*}_{i,j} } } \geq \frac{e^{-\left|g_{i,j+1}^{\ell,n}-g_{i,j}^{\ell,n}\right|}}{c_{\max}}.
		$$
		Thus, we have
		\begin{equation}\label{I1/I2}
			\frac{I_1}{2I_2} \geq \frac{1}{2C \kappa c_{\max}} e^{-\max_{i,j,\ell} \left\{\left|g_{i+1,j}^{\ell,n}-g_{i,j}^{\ell,n}\right|,\left|g_{i,j+1}^{\ell,n}-g_{i,j}^{\ell,n}\right|\right\}}.
		\end{equation}
		Note that the estimate~\reff{I1/I2} also works for the geometric mean and arithmetic mean approximations mentioned in Remark~\ref{Means}.
		Therefore, a choice of
	\[
		\Delta t^{*}=\frac{ 2 \kappa\varepsilon_{\min }^{3}}{\varepsilon_{\max }^2  c_{\max} \sum_{\ell=1}^{M} \left|q^\ell \right|^{2}} {\rm exp}\left[-  \max _{i,j,\ell} \left(\left| dg^{\ell,n}_{i+\frac{1}{2},j} \right|,\left| dg^{\ell,n}_{i,j+\frac{1}{2}} \right| \right) \right]
	\]
completes the proof.		
	\end{proof}

\section*{Acknowledgements}
This work is supported by the CAS AMSS-PolyU Joint Laboratory of Applied Mathematics. Z. Qiao's work is partially supported by the Hong Kong Research Grants Council (RFS Project No. RFS2021-5S03 and GRF project No. 15302919) and the Hong Kong Polytechnic University internal grant No. 1-9B7B. The work of Z. Xu and Q. Yin is partially supported by  NSFC (grant No. 12071288), Science and Technology Commission of Shanghai Municipality (grant No. 20JC1414100 and 21JC1403700), the Strategic Priority Research Program of CAS (grant No. XDA25010403) and the HPC center of Shanghai Jiao Tong University. S. Zhou's work is partially supported by the National Natural Science Foundation of China 12171319 and Science and Technology Commission of Shanghai Municipality (grant No. 21JC1403700).

\bibliographystyle{plain}


\begin{thebibliography}{10}

\bibitem{BSD:PRE:2009}
M.~Baptista, R.~Schmitz, and B.~D{\"u}nweg.
\newblock Simple and robust solver for the {Poisson-Boltzmann} equation.
\newblock {\em Phys. Rev. E}, 80:016705, 2009.

\bibitem{BAO:PRL:1997}
I.~Borukhov, D.~Andelman, and H.~Orland.
\newblock Steric effects in electrolytes: A modified {Poisson-Boltzmann}
  equation.
\newblock {\em Phys. Rev. Lett.}, 79:435--438, 1997.

\bibitem{Ding2022con}
J.~Ding, C.~Wang, and S.~Zhou.
\newblock Convergence analysis of structure-preserving numerical methods based
  on {Slotboom} transformation for the {Poisson--Nernst--Planck} equations.
\newblock {\em arXiv preprint arXiv:2202.10931}, 2022.

\bibitem{Ding19JCP}
J.~Ding, Z.~Wang, and S.~Zhou.
\newblock Positivity preserving finite difference methods for
  {Poisson--Nernst--Planck} equations with steric interactions: {A}pplication
  to slit-shaped nanopore conductance.
\newblock {\em J. Comput. Phys.}, 397:108864, 2019.

\bibitem{Ding2020JCP}
J.~Ding, Z.~Wang, and S.~Zhou.
\newblock Structure-preserving and efficient numerical methods for ion
  transport.
\newblock {\em J. Comput. Phys.}, 418:109597, 2020.

\bibitem{E:CP:98}
B.~Eisenberg.
\newblock Ionic channels in biological membranes-electrostatic analysis of a
  natural nanotube.
\newblock {\em Contemp. Phys.}, 39(6):447--466, 1998.

\bibitem{FH:PRE:2014}
F.~Fahrenberger and C.~Holm.
\newblock Computing the {Coulomb} interaction in inhomogeneous dielectric media
  via a local electrostatics lattice algorithm.
\newblock {\em Phys. Rev. E}, 90(6):063304, 2014.

\bibitem{FXH:JCP:14}
F.~Fahrenberger, Z.~Xu, and C.~Holm.
\newblock Simulation of electric double layers around charged colloids in
  aqueous solution of variable permittivity.
\newblock {\em J. Chem. Phys.}, 141(6):064902, 2014.

\bibitem{Farrel1991SG}
P.~A. Farrell and E.~C. Gartland~Jr.
\newblock On the {Scharfetter-Gummel} discretization for drift-diffusion
  continuity equations.
\newblock In {\em Computational methods for boundary and interior layers in
  several dimensions}, pages 51--79. Boole Press, 1991.

\bibitem{Gao2017JSC}
H.~Gao and D.~He.
\newblock Linearized conservative finite element methods for the
  {Nernst-Planck-Poisson} equations.
\newblock {\em J. Sci. Comput.}, 72(3):1269--1289, 2017.

\bibitem{HorngLinLiuBob_JPCB12}
T.~Horng, T.~Lin, C.~Liu, and R.~Eisenberg.
\newblock {PNP} equations with steric effects: a model of ion flow through
  channels.
\newblock {\em J. Phys. Chem. B}, 116(37):11422--11441, 2012.

\bibitem{hu2020NM}
J.~Hu and X.~Huang.
\newblock A fully discrete positivity-preserving and energy-dissipative finite
  difference scheme for {Poisson-Nernst-Planck} equations.
\newblock {\em Numer. Math.}, pages 1--39, 2020.

\bibitem{KBA:PRE:2007}
M.~S. Kilic, M.~Z. Bazant, and A.~Ajdari.
\newblock Steric effects in the dynamics of electrolytes at large applied
  voltages. {II.} {M}odified {P}oisson--{N}ernst--{P}lanck equations.
\newblock {\em Phys. Rev. E}, 75:021503, 2007.

\bibitem{Lee_JMP2014}
C.C. Lee.
\newblock The charge conserving {Poisson-Boltzmann} equations: {Existence},
  uniqueness, and maximum principle.
\newblock {\em J. Math. Phys.}, 55(5):051503, 2014.

\bibitem{Lee_NonL2010}
C.C. Lee, H.~Lee, Y.~Hyon, T.C. Lin, and C.~Liu.
\newblock New {Poisson--Boltzmann} type equations: one-dimensional solutions.
\newblock {\em Nonlinearity}, 24(2):431, 2010.

\bibitem{LiuChun2020positivity}
C.~Liu, C.~Wang, S.~M. Wise, X.~Yue, and S.~Zhou.
\newblock A positivity-preserving, energy stable and convergent numerical
  scheme for the {Poisson-Nernst-Planck} system.
\newblock {\em Math. Comput.}, 90:2071--2106, 2021.

\bibitem{LM2020}
H.~Liu and W.~Maimaitiyiming.
\newblock Efficient, positive, and energy stable schemes for {multi-D
  Poisson--Nernst--Planck} systems.
\newblock {\em J. Sci. Comput.}, 87(3):1--36, 2021.

\bibitem{LW:JCP:14}
H.~Liu and Z.~Wang.
\newblock A free energy satisfying finite difference method for
  {Poisson-Nernst-Planck} equations.
\newblock {\em J. Comput. Phys.}, 268:363 -- 376, 2014.

\bibitem{liu2017DG}
H.~Liu and Z.~Wang.
\newblock A free energy satisfying discontinuous {Galerkin} method for
  one-dimensional {Poisson--Nernst--Planck} systems.
\newblock {\em J. Comput. Phys.}, 328:413--437, 2017.

\bibitem{citeone}
J.-G. Liu, Y.~Tang, and Y.~Zhao.
\newblock On the equilibrium of the {Poisson-Nernst-Planck-Bikermann} model
  equipping with the steric and correlation effects.
\newblock {\em arXiv preprint arXiv:2201.01423}, 2022.

\bibitem{citetwo}
J.-G. Liu, J.~Wang, Z.~Zhou, and Y.~Zhao.
\newblock Field model for complex ionic fluids: analytical properties and
  numerical investigation.
\newblock {\em Commun. Comput. Phys.}, 30(3):874--902, 2019.

\bibitem{LJX:SIAP:2018}
P.~Liu, X.~Ji, and Z.~Xu.
\newblock {Modified Poisson--Nernst--Planck} model with accurate {Coulomb}
  correlation in variable media.
\newblock {\em SIAM J. Appl. Math.}, 78:226--245, 2018.

\bibitem{LZ:BJ:2011}
B.~Lu and Y.~Zhou.
\newblock {Poisson-Nernst-Planck} equations for simulating biomolecular
  diffusion-reaction processes {II}: Size effects on ionic distributions and
  diffusion-reaction rates.
\newblock {\em Biophys. J.}, 100:2475--2485, 2011.

\bibitem{MXZ:SIAP:2021}
M.~Ma, Z.~Xu, and L.~Zhang.
\newblock Modified {Poisson--Nernst--Planck} model with {Coulomb} and
  hard-sphere correlations.
\newblock {\em SIAM J. Appl. Math.}, 81:1645--1667, 2021.

\bibitem{M:JCP:2002}
A.~C. Maggs.
\newblock Dynamics of a local algorithm for simulating {C}oulomb interactions.
\newblock {\em J. Chem. Phys.}, 117(5):1975--1981, 2002.

\bibitem{MR:PRL:2002}
A.~C. Maggs and V.~Rossetto.
\newblock Local simulation algorithms for {C}oulomb interactions.
\newblock {\em Phys. Rev. Lett.}, 88:196402, 2002.

\bibitem{2012semiconductor}
P.~A. Markowich, C.~A. Ringhofer, and C.~Schmeiser.
\newblock {\em {Semiconductor equations}}.
\newblock Springer Science \& Business Media, 2012.

\bibitem{Metti:JCP:2016}
M.~S. Metti, J.~Xu, and C.~Liu.
\newblock Energetically stable discretizations for charge transport and
  electrokinetic models.
\newblock {\em J. Comput. Phys.}, 306:1--18, 2016.

\bibitem{PD:JPCM:2004}
I.~Pasichnyk and B.~D{\"u}nweg.
\newblock Coulomb interactions via local dynamics: {A} molecular-dynamics
  algorithm.
\newblock {\em J. Phys: Condens. Matter}, 16:S3999, 2004.

\bibitem{Qian2021positive}
Y.~Qian, C.~Wang, and S.~Zhou.
\newblock A positive and energy stable numerical scheme for the
  {Poisson--Nernst--Planck--Cahn--Hilliard} equations with steric interactions.
\newblock {\em J. Comput. Phys.}, 426:109908, 2021.

\bibitem{QianWangZhou_JCP19}
Y.~Qian, Z.~Wang, and S.~Zhou.
\newblock A conservative, free energy dissipating, and positivity preserving
  finite difference scheme for multi-dimensional nonlocal {Fokker--Planck}
  equation.
\newblock {\em J. Comput. Phys.}, 386:22--36, 2019.

\bibitem{Qiao2022ANP}
Z.~Qiao, Z.~Xu, Q.~Yin, and S.~Zhou.
\newblock A {Maxwell-Amp{\`e}re Nernst-Planck} framework for modeling charge
  dynamics.
\newblock {\em arXiv preprint arXiv:2202.07366}, 2022.

\bibitem{Qing:Langmuir2020}
L.~Qing, J.~Lei, T.~Zhao, G.~Qiu, M.~Ma, Z.~Xu, and S.~Zhao.
\newblock Effects of kinetic dielectric decrement on ion diffusion and
  capacitance in electrochemical systems.
\newblock {\em Langmuir}, 36(15):4055--4064, 2020.

\bibitem{RM:JCP:2004MC}
J.~Rottler and A.~C. Maggs.
\newblock A continuum, {O (N) Monte Carlo} algorithm for charged particles.
\newblock {\em J. Chem. Phys.}, 120(7):3119--3129, 2004.

\bibitem{RM:PRL:2004}
J.~Rottler and A.~C. Maggs.
\newblock Local molecular dynamics with {Coulombic} interactions.
\newblock {\em Phys. Rev. Lett.}, 93(17):170201, 2004.

\bibitem{Scharfetter1969IEEE}
D.~L. Scharfetter and H.~K. Gummel.
\newblock Large-signal analysis of a silicon read diode oscillator.
\newblock {\em IEEE Trans. Elec. Dev.}, 16(1):64--77, 1969.

\bibitem{Schoch:RMP:08}
R.~B. Schoch, J.~Han, and P.~Renaud.
\newblock Transport phenomena in nanofluidics.
\newblock {\em Rev. Mod. Phys.}, 80:839--883, 2008.

\bibitem{SJXJ2021}
J.~Shen and J.~Xu.
\newblock Unconditionally positivity preserving and energy dissipative schemes
  for {Poisson--Nernst--Planck} equations.
\newblock {\em Numer. Math.}, pages 1--27, 2021.

\bibitem{SWZ_CMS18}
F.~Siddiqua, Z.~Wang, and S.~Zhou.
\newblock A modified {P}oisson--{N}ernst--{P}lanck model with excluded volume
  effect: theory and numerical implementation.
\newblock {\em Commun. Math. Sci.}, 16 (1):251--271, 2018.

\bibitem{Wan_PRX2014}
L.~Wan, S.~Xu, M.~Liao, C.~Liu, and P.~Sheng.
\newblock Self-consistent approach to global charge neutrality in
  electrokinetics: A surface potential trap model.
\newblock {\em Phys. Rev. X}, 4(1):011042, 2014.

\bibitem{Wang:PRE:2010}
Zhen-Gang Wang.
\newblock Fluctuation in electrolyte solutions: The self energy.
\newblock {\em Phys. Rev. E}, 81:021501, 2010.

\bibitem{Wu:JCP:2002}
Y.-X. Yu and J.-Z. Wu.
\newblock Structures of hard-sphere fluids from a modified fundamental-measure
  theory.
\newblock {\em J. Chem. Phys.}, 117(22):10156--10164, 2002.

\bibitem{ZWL:PRE:2011}
S.~Zhou, Z.~Wang, and B.~Li.
\newblock Mean-field description of ionic size effects with non-uniform ionic
  sizes: {A} numerical approach.
\newblock {\em Phys. Rev. E}, 84:021901, 2011.

\end{thebibliography}

\end{document}